\documentclass[11pt]{amsart}
\usepackage[a4paper,margin=27mm]{geometry}
\usepackage[T1]{fontenc}
\usepackage{lmodern,microtype,amsmath,amssymb,mathtools,booktabs,array,tabularx,enumitem}
\usepackage[hidelinks]{hyperref}
\hypersetup{
 pdftitle={Three irrationality results for the dilogarithm},
 pdfauthor={Thomas Prellberg},
 pdfsubject={Irrationality of Li2(-1/4), Li2(1/5), and Li2(-1/3)},
 pdfkeywords={dilogarithm, irrationality, Rhin-Viola, Viola-Zudilin, permutation group method}
}
\newcommand{\Li}{\operatorname{Li}}
\newcommand{\Log}{\operatorname{Log}}
\newcommand{\Z}{\mathbb Z}
\newcommand{\Q}{\mathbb Q}

\newcommand{\dd}{\,d}
\newcommand{\Res}{\operatorname{Res}}
\newcommand{\ind}{\mathbf 1}
\newcommand{\vp}{v_p}
\newcommand{\lcm}{\operatorname{lcm}}

\newcommand{\Bfun}{\mathrm B}

\newcommand{\dist}{\operatorname{dist}}
\theoremstyle{plain}
\newtheorem{theorem}{Theorem}[section]
\newtheorem{lemma}[theorem]{Lemma}
\newtheorem{proposition}[theorem]{Proposition}

\newtheorem*{corollary*}{Corollary}
\theoremstyle{remark}\newtheorem{remark}[theorem]{Remark}
\numberwithin{equation}{section}
\allowdisplaybreaks
\title[Three irrationality results for the dilogarithm]{Three Irrationality Results for the Dilogarithm:\\Rhin--Viola and Viola--Zudilin Constructions at $-1/4$, $1/5$, and $-1/3$}
\author{Thomas Prellberg}
\address{School of Mathematical Sciences, Queen Mary University of London, Mile End Road, London E1 4NS, United Kingdom}
\email{t.prellberg@qmul.ac.uk}
\date{7 September 2026}
\begin{document}
\begin{abstract}
We prove the irrationality of
\[
 \Li_2(-1/4),\qquad \Li_2(1/5),\qquad \Li_2(-1/3).
\]
The three arguments form a natural progression.  The first is an endpoint completion of the five-parameter Rhin--Viola method.  Rhin and Viola's 2019 treatment already supplies the negative-argument continuation, permutation invariance and factorial divisor; at $z=-4$ we use three bounded shifts, an elementary fixed-contour estimate and a four-term recurrence to remove the remaining nonvanishing problem in the complex-saddle regime.  For $1/5$ we pass to the six-parameter Viola--Zudilin family and apply the Rhin--Viola factorial transformations term by term inside two binomial expansions of the same integral.  This inherited divisor crosses the arithmetic threshold that the unrefined five-parameter construction does not reach in our computations.  The same crossbreed, combined with the negative-argument continuation, proves the result at $-1/3$.  All proof-critical finite inequalities are accompanied by exact executable certificates.  Supplementary computations propagate $10{,}000$ exact primitive coefficient pairs for each of $1/5$ and $-1/3$, with complete gcd removal and no failure of the certified smallness or adjacent nonproportionality checks.
\end{abstract}
\maketitle

\section{Introduction and statement of results}\label{sec:intro}
For $|w|<1$ write
\[
 \Li_2(w)=\sum_{m\ge1}\frac{w^m}{m^2}.
\]
Hata proved the irrationality of $\Li_2(1/s)$ for all integers $s\le -5$ and $s\ge7$ \cite{Hata}.  Rhin and Viola developed the five-parameter permutation-group method and reached $s=6$ on the positive side \cite{RV05}.  Marcovecchio subsequently treated linear independence problems for polylogarithms at algebraic points \cite{Marc16}.  In 2019 Rhin and Viola extended their construction explicitly to negative rational $z$, including analytic continuation, the negative branch of the logarithm, permutation invariance and the factorial divisor, and obtained linear independence results for $1,\Li_1(1/z),\Li_2(1/z)$ for integer $z\le-5$ \cite{RV19}.  Viola and Zudilin introduced a six-parameter extension yielding stronger multi-value linear-independence results in more distant ranges \cite{VZ18}.

The integer arguments $s\in\{-4,-3,-2,2,3,4,5\}$ are recorded as unresolved individual irrationality cases in the account of Calegari, Dimitrov and Tang \cite[Remark~2.8.2 and \S15.2]{CDT}.  We settle three of them.

\begin{theorem}\label{thm:main}
The three numbers
\[
 \Li_2(-1/4),\qquad \Li_2(1/5),\qquad \Li_2(-1/3)
\]
are irrational.
\end{theorem}

The point of combining the proofs is that their relationship is clearer together than separately.
\begin{enumerate}[label=(\roman*),leftmargin=2.2em]
\item At $-1/4$ no new prime divisor is needed.  A direct check of the Rhin--Viola transformations shows that the ray $(10,8,8,5,10)$ is $\Phi$-equivalent to their 2019 example $(10,8,7,6,10)$; in particular the two have the same exponent multiset.  The published negative-$z$ arithmetic therefore applies unchanged.  What remains at this endpoint is a nonvanishing issue: the relevant asymptotic contribution comes from a conjugate complex pair, whereas the qualitative criterion used in \cite{RV19} assumes a strict real-saddle ordering $c_3<c_0<c_1$.  We bypass this by working with three bounded shifts and a recurrence.  The resulting proof establishes irrationality only; it does not claim the stronger linear independence with $\Li_1(-1/4)$ obtained by Rhin--Viola in their range.
\item At $1/5$ the five-parameter construction no longer crosses the threshold in our extensive searches.  The new arithmetic ingredient is to expand one six-parameter Viola--Zudilin integral in two ways as sums of five-parameter Rhin--Viola forms and to choose the best factorial transform separately for each summand and prime.  The resulting common divisor is the essential new ingredient.
\item At $-1/3$ the same six-parameter arithmetic crossbreed is combined with the negative-argument continuation.  This is the most delicate of the three proofs numerically: the certified final exponential margin is just over $1/20$.
\end{enumerate}
The parameter searches mentioned in (ii) and (iii) are only motivation; no optimality assertion for the five-parameter family is used in any proof.

The proof-critical computations are exact.  Separately implemented checks reproduce the coefficient identities, divisor witnesses, recurrences, contour inequalities and endpoint signs.  Section~\ref{sec:verification-combined} records the verification boundary and a $10{,}000$-term stress test for each of the two six-parameter constructions.

\section{The Rhin--Viola endpoint at $-1/4$}\label{sec:quarter}
The negative-argument framework in this section is already part of \cite{RV19}; our additions are the explicit fixed-contour bound in the chosen normalization and the bounded-shift nonvanishing argument.
For $n\ge1$ and $r=0,1,2$ put
\begin{equation}\label{q:eq:shifted-tuple}
  \boldsymbol\nu_{n,r}:=(10n+r,8n,8n,5n,10n).
\end{equation}
\subsection{Five-parameter linear forms and arithmetic normalisation}\label{q:sec:rv}

We recall the part of \cite{RV05} that we need.  For $z>1$ and
nonnegative integers $h,j,k,l,m$, put
\begin{equation}\label{q:eq:I0}
 I_z^{(0)}(h,j,k,l,m)
 =z^{-l-m}\int_0^1\!\int_0^1
 \frac{x^j(1-x)^h y^k(1-y)^l}
      {(x(1-y)+yz)^{j+k-m+1}}\dd y\dd x .
\end{equation}
For each fixed $x$, let $\gamma_{x,z}$ be a small positively oriented circle
about the pole
\[
  y_0=\frac{x}{x-z}.
\]
Define
\begin{equation}\label{q:eq:I1}
 I_z^{(1)}(h,j,k,l,m)
 =z^{-l-m}\int_0^1\frac{1}{2\pi i}
 \oint_{\gamma_{x,z}}
 \frac{x^j(1-x)^h y^k(1-y)^l}
      {(x(1-y)+yz)^{j+k-m+1}}\dd y\dd x
\end{equation}
and
\begin{equation}\label{q:eq:I}
  I_z(h,j,k,l,m)=I_z^{(0)}(h,j,k,l,m)
                  -(\log z)I_z^{(1)}(h,j,k,l,m).
\end{equation}
These are equations (2.1), (2.2) and (2.4) of
\cite{RV05}.

Write $d_N=\lcm(1,\ldots,N)$, with $d_0=1$.  Rhin and Viola associate to a
tuple the nine integers
\begin{equation}\label{q:eq:Sset}
 \mathcal S=\{h,j,k,l,m,l+m-j,m+h-k,h+j-l,j+k-m\}.
\end{equation}
Their group-invariant l.c.m. parameters are
\[
 M=\max\mathcal S
\]
and a second maximum $N$ defined in \cite[(4.2)]{RV05}; for the
tuples used below the values will be immediate.  Their Theorem 2.1 and the
group-invariant enlargement in \cite[\S4]{RV05} give integer
polynomials after multiplication by $d_Md_N$ and suitable powers of $z$ and
$z-1$.

\begin{proposition}[Rhin--Viola, specialized]\label{q:prop:rv-specialized}
For every $n\ge1$, $r\in\{0,1,2\}$ and real $z>1$, there exist polynomials
$P_{n,r},Q_{n,r}\in\Z[z]$ such that
\begin{equation}\label{q:eq:rv-specialized}
 \begin{aligned}
 &d_{13n+r}d_{12n+r}z^{16n}(z-1)^{3n-r}
 I_z(\boldsymbol\nu_{n,r})\\
 &\hspace{35mm}=P_{n,r}(z)-Q_{n,r}(z)\Li_2(1/z).
 \end{aligned}
\end{equation}
\end{proposition}

\begin{proof}
For the tuple \eqref{q:eq:shifted-tuple}, the set \eqref{q:eq:Sset} is
\[
 \{10n+r,8n,8n,5n,10n,7n,12n+r,13n+r,6n\}.
\]
Thus the group-invariant maxima of \cite[(4.1)--(4.2)]{RV05} are
\[
 M=13n+r,\qquad N=12n+r.
\]
The two power parameters in \cite[(2.9)]{RV05} are
\[
 \alpha=\max\{j+k,k+l,l+m\}=16n,
 \qquad
 \beta=\max\{0,k+l-h\}=3n-r.
\]
Equation \eqref{q:eq:rv-specialized} is therefore the stated specialization of
\cite[Theorem 2.1 and \S4]{RV05}.
\end{proof}

It is convenient to enlarge all three forms to one common normalization.
Since $d_{13n+r}\mid d_{13n+2}$ and $d_{12n+r}\mid d_{12n+2}$, multiplying
\eqref{q:eq:rv-specialized} by the corresponding integer ratios and by
$(z-1)^r$ gives
\begin{equation}\label{q:eq:common-rv}
 d_{13n+2}d_{12n+2}z^{16n}(z-1)^{3n}
 I_z(\boldsymbol\nu_{n,r})
 \in \Z[z]+\Z[z]\Li_2(1/z)
\end{equation}
for $r=0,1,2$.

\subsection{The explicit $z=-4$ realization}\label{q:sec:continuation}

We continue $z$ from $4$ to $-4$ along the upper semicircle
\begin{equation}\label{q:eq:zpath}
  z=4e^{it},\qquad 0\le t\le\pi,
\end{equation}
using the branch $\Log z$ with $0\le\arg z\le\pi$.  The following lemma fixes
the side of the moving pole and the logarithmic term; this is the point at
which negative arguments require more care than the positive theory.

\begin{lemma}[Continuation of the inner contour]\label{q:lem:continuation}
Fix $0<x\le1$.  Along \eqref{q:eq:zpath}, the pole
\[
 y_0(z)=\frac{x}{x-z}
\]
lies in the upper half-plane for $0<t<\pi$.  Hence the continuation of the
original interval $[0,1]$ passes below the pole.  At $z=-4$ it may be
deformed to a path $\Gamma_x$ defined below.  Under
\[
 u=x+(z-x)y,
\]
the lower detour in the $y$-plane becomes an upper detour about $u=0$, and
its logarithmic contribution is
\begin{equation}\label{q:eq:log-contribution}
  \Log(-4)-\log x=\log4+i\pi-\log x.
\end{equation}
Moreover, the continuation of $\Li_2(1/z)$ along \eqref{q:eq:zpath} ends at the
ordinary real value $\Li_2(-1/4)$.
\end{lemma}

\begin{proof}
Write $z=u+iv$ with $v>0$.  Then
\[
 y_0(z)=\frac{x(x-u+iv)}{(x-u)^2+v^2},
\]
so $\operatorname{Im}y_0(z)>0$.  The transported interval therefore stays
below the moving pole.  At $z=-4$ write the lower detour as
\[
 y-a=\rho e^{i\vartheta},\qquad -\pi\le\vartheta\le0.
\]
Since $u=(z-x)(y-a)$ and $z-x<0$, its image has continuously chosen
argument $0\le\arg u\le\pi$: it is an upper detour about $u=0$.
Thus integrating $du/u$ from $u=x>0$ to $u=-4$ along the continued path
gives \eqref{q:eq:log-contribution}.

Finally, $1/z=\frac14e^{-it}$ remains throughout the open unit disc.  The
power-series branch of the dilogarithm therefore continues without meeting a
singularity and ends at $-1/4$.
\end{proof}

For the rest of the paper put
\begin{equation}\label{q:eq:delta-a-rho}
 \Delta(x,y)=x-(x+4)y,\qquad
 a=\frac{x}{x+4},\qquad
 \rho=\frac27.
\end{equation}
Thus $0<a\le1/5<\rho$ and $a+\rho<1$.  Let $\Gamma_x$ run from $0$ along the
real axis to $a-\rho$, then along the lower semicircle centred at $a$ with
radius $\rho$, and finally along the real axis from $a+\rho$ to $1$.  Let
$C_x$ be the positively oriented full circle with the same centre and radius.
Define
\begin{equation}\label{q:eq:Jdef}
\begin{split}
 J_n(x)={}&\int_{\Gamma_x}
  \frac{y^{8n}(1-y)^{5n}}{\Delta(x,y)^{6n+1}}\dd y\\
 &-\frac{\log4+i\pi}{2\pi i}
  \oint_{C_x}\frac{y^{8n}(1-y)^{5n}}
  {\Delta(x,y)^{6n+1}}\dd y,
\end{split}
\end{equation}
then
\begin{equation}\label{q:eq:g-mu-F}
 g_n(x)=x^{8n}J_n(x),\qquad
 \mu_H=\int_0^1(1-x)^H g_n(x)\dd x,
 \qquad
 F_{n,r}=\mu_{10n+r}.
\end{equation}

\begin{proposition}\label{q:prop:F-I}
For $r=0,1,2$,
\begin{equation}\label{q:eq:F-I}
 F_{n,r}=(-4)^{15n}I_{-4}(\boldsymbol\nu_{n,r}),
\end{equation}
where the right-hand side denotes continuation along \eqref{q:eq:zpath}.
Furthermore $g_n$ is real on $(0,1)$ and
\begin{equation}\label{q:eq:g-zero}
 g_n(x)=c_nx^{8n}+O(x^{8n+1})\qquad(x\to0^+),
\end{equation}
with
\begin{equation}\label{q:eq:cn}
 c_n=(-4)^{-6n-1}\frac{(2n-1)!(5n)!}{(7n)!}\ne0.
\end{equation}
\end{proposition}

\begin{proof}
The identity \eqref{q:eq:F-I} follows directly from
\eqref{q:eq:I0}--\eqref{q:eq:I}, Lemma \ref{q:lem:continuation}, and the factor
$z^{-l-m}=(-4)^{-15n}$.

For a real expression, put $z=-4$ and make the substitution
$u=x+(z-x)y$.  If
\[
 (u-x)^{8n}(z-u)^{5n}=\sum_{s=0}^{13n}e_s(x,z)u^s,
\]
then integration of the monomials and subtraction of the $\Log z$ residue
give
\begin{equation}\label{q:eq:J-real}
 J_n(x)=\frac1{(z-x)^{13n+1}}
 \left(
 -e_{6n}(x,z)\log x
 +\sum_{s\ne6n}e_s(x,z)
   \frac{z^{s-6n}-x^{s-6n}}{s-6n}
 \right).
\end{equation}
The expression is real for $0<x<1$.

The inner integral before the residue subtraction is, up to a nonzero
constant factor,
\begin{equation}\label{q:eq:hypergeom-inner}
 z^{-6n-1}\Bfun(5n+1,8n+1)
 {}_2F_1\left(6n+1,5n+1;13n+2;1-\frac{x}{z}\right).
\end{equation}
The Gauss equation has exponents $0$ and $2n$ at $x=0$; consequently the
logarithmic coefficient in \eqref{q:eq:J-real} is $O(x^{2n})$, and its rational
part is analytic at zero.  Its constant term is obtained either from
\eqref{q:eq:J-real} or by a binomial expansion:
\begin{align*}
 J_n(0)
 &=z^{-6n-1}\sum_{b=0}^{5n}
   \frac{(-1)^b\binom{5n}{b}}{2n+b}\\
 &=z^{-6n-1}\Bfun(2n,5n+1)
 =z^{-6n-1}\frac{(2n-1)!(5n)!}{(7n)!}.
\end{align*}
Taking $z=-4$ and multiplying by $x^{8n}$ proves
\eqref{q:eq:g-zero}--\eqref{q:eq:cn}.
\end{proof}

\subsection{A fixed-contour exponential estimate}\label{q:sec:analytic}

The point of the fixed contour is that only an upper bound is needed.  Put
\begin{equation}\label{q:eq:fdef}
 f(x,y)=\frac{x^8(1-x)^{10}y^8(1-y)^5}{\Delta(x,y)^6}.
\end{equation}
The integrand contributing to $F_{n,r}$ is
\[
 (1-x)^r\frac{f(x,y)^n}{\Delta(x,y)}.
\]
We first control the two straight parts of $\Gamma_x$ by their adjacent
circular endpoints.

\begin{lemma}\label{q:lem:straight}
For fixed $x\in(0,1]$, the absolute value of the full $y$-dependent integrand
is increasing on the left straight segment of $\Gamma_x$ and decreasing on
the right straight segment.  Hence its maximum on either segment occurs at
the endpoint on $|y-a|=\rho$.
\end{lemma}

\begin{proof}
On the left write $y=-v$, $0<v\le\rho-a$.  Apart from factors independent of
$v$, the absolute value is
\[
 \frac{v^{8n}(1+v)^{5n}}{(a+v)^{6n+1}}.
\]
Its logarithmic derivative satisfies
\[
 \frac{8n}{v}+\frac{5n}{1+v}-\frac{6n+1}{a+v}
 \ge \frac{2n-1}{v}+\frac{5n}{1+v}>0.
\]

On the right, the logarithmic derivative of the $n$th-power part
$y^8(1-y)^5/(y-a)^6$ has the sign of
\[
 q_a(y)=-7y^2+(2+13a)y-8a.
\]
With $b=a+2/7$,
\[
 q_a(b)=6a^2-\frac{44}{7}a\le0,
 \qquad
 b-\frac{2+13a}{14}=\frac17+\frac a{14}>0.
\]
Thus the vertex of the concave quadratic lies to the left of $b$, so
$q_a(y)\le0$ for $y\ge b$.  The additional factor $(y-a)^{-1}$ is also
decreasing, which proves the claim.
\end{proof}

On the circle, put $U=|y|^2$ and $V=|1-y|^2$.  Since $|y-a|=\rho$,
\begin{equation}\label{q:eq:UVidentity}
 (1-a)U+aV=a(1-a)+\rho^2=:T.
\end{equation}
Weighted arithmetic--geometric mean gives
\begin{equation}\label{q:eq:weighted-AMGM}
 U^8V^5
 \le \frac{8^85^5}{13^{13}}\frac{T^{13}}{(1-a)^8a^5}.
\end{equation}
Substitution of \eqref{q:eq:delta-a-rho} into \eqref{q:eq:weighted-AMGM} yields
\begin{equation}\label{q:eq:Fcal}
 |f(x,y)|^2\le\mathcal F(x),
\end{equation}
where
\begin{equation}\label{q:eq:Fcal-explicit}
 \boxed{
 \mathcal F(x)=
 \frac{2^{22}5^5}{13^{13}7^{14}}
 \frac{x^{11}(1-x)^{20}(x^2+57x+16)^{13}}{(x+4)^{25}}.}
\end{equation}
Let
\begin{equation}\label{q:eq:Bdef}
 \mathcal B=\max_{0\le x\le1}\sqrt{\mathcal F(x)}.
\end{equation}
On the circle, $|\Delta|=(x+4)\rho\ge4\rho$.  Lemma
\ref{q:lem:straight} transfers the same endpoint bound to the straight pieces.
The total length of $\Gamma_x$ is $1-2a+\pi\rho\le1+\pi\rho$.  Therefore
\begin{equation}\label{q:eq:Fbound}
 |F_{n,r}|\le C_0\mathcal B^n\qquad(r=0,1,2),
\end{equation}
with the fixed constant
\begin{equation}\label{q:eq:C0}
 C_0=\frac{1+\pi\rho}{4\rho}
     +\frac{|\log4+i\pi|}{4}.
\end{equation}

We next give a deliberately coarse rational bound for $\mathcal B$.

\begin{lemma}\label{q:lem:B-rational}
One has
\begin{equation}\label{q:eq:Bsimple}
  \mathcal B<7\cdot10^{-11}.
\end{equation}
\end{lemma}

\begin{proof}
The sign of $\mathcal F'(x)$ on $(0,1)$ is the sign of
\begin{equation}\label{q:eq:quartic}
 p(x)=-32x^4-1299x^3-10037x^2+3264x+704.
\end{equation}
Descartes' rule of signs shows that $p$ has exactly one positive root; since
$p(0)>0>p(1)$, that root is the unique maximizer of $\mathcal F$ in $(0,1)$.
Direct rational evaluation gives
\begin{equation}\label{q:eq:xbracket}
 a_0:=\frac{4531215}{10^7}<x_*<
 b_0:=\frac{4531216}{10^7}.
\end{equation}
Because the increasing and decreasing factors in \eqref{q:eq:Fcal-explicit}
are monotone on this interval,
\begin{equation}\label{q:eq:Ubound}
 \mathcal B^2\le U:=
 \frac{2^{22}5^5}{13^{13}7^{14}}
 \frac{b_0^{11}(1-a_0)^{20}(b_0^2+57b_0+16)^{13}}
      {(a_0+4)^{25}}.
\end{equation}
A direct integer cross-multiplication gives
\[
 U<\frac{49}{10^{22}},
\]
which is equivalent to \eqref{q:eq:Bsimple}.
\end{proof}

\subsection{The common arithmetic divisor}\label{q:sec:arithmetic}

Set
\begin{equation}\label{q:eq:theta}
  \theta=\Li_2(-1/4).
\end{equation}
Continuation of the polynomial identity \eqref{q:eq:common-rv} along
\eqref{q:eq:zpath}, followed by Proposition \ref{q:prop:F-I}, gives
\begin{equation}\label{q:eq:predivisor}
 d_{13n+2}d_{12n+2}4^n5^{3n}F_{n,r}
 \in\Z+\Z\theta,
 \qquad r=0,1,2.
\end{equation}
The power of $4$ in \eqref{q:eq:predivisor} is worth displaying explicitly:
\begin{equation}\label{q:eq:scaling-one-line}
 \frac{(-4)^{16n}(-5)^{3n}}{(-4)^{15n}}
 =4^n5^{3n}.
\end{equation}

We now use the factorial savings of \cite[\S\S3--4]{RV05}.  For
$w\in\mathbb R$ define
\[
 H=\lfloor10w\rfloor,\quad J=K=\lfloor8w\rfloor,
 \quad L=\lfloor5w\rfloor,\quad M_0=\lfloor10w\rfloor,
\]
\[
 A=\lfloor7w\rfloor,\quad B_0=\lfloor12w\rfloor,
 \quad C=\lfloor13w\rfloor,\quad D=\lfloor6w\rfloor,
\]
and
\begin{equation}\label{q:eq:e-def}
\begin{split}
 e(w)=\max\{&0,
 H+J-B_0-D,
 H+M_0-C-A,
 K+L-A-D,\\
 &H+J+K-C-A-D,
 H+L+M_0-B_0-D-A\}.
\end{split}
\end{equation}
Every expression in the maximum is balanced, so $e(w+1)=e(w)$.

We use the same factorial divisor at the integer point $z=-4$.  The transformation formulae originate in \cite[\S\S3--4]{RV05}, and \cite[\S2]{RV19} proves that the polynomial identities, the $\Phi$-invariance and the resulting divisor continue unchanged to $z<0$.  Thus no new arithmetic input is introduced at this point; the calculation below is the specialization of the published Rhin--Viola divisor to the chosen ray.

\begin{lemma}[Specialized Rhin--Viola divisor]\label{q:lem:rv-divisor}
For the unshifted member $\boldsymbol\nu_{n,0}$ and any prime $p^2>13n$,
both integer coefficients of the Rhin--Viola linear form are divisible by
$p^{e(n/p)}$.  Moreover $e$ takes only the values $0,1,2$.
\end{lemma}

\begin{proof}
For the base ray $(10,8,8,5,10)$, the four complementary parameters in
\eqref{q:eq:Sset} are
\[
 l+m-j=7,\qquad m+h-k=12,\qquad h+j-l=13,
 \qquad j+k-m=6.
\]
Write the five nonzero candidates in \eqref{q:eq:e-def} as
$D_1,\ldots,D_5$ in their displayed order.  They are exactly the
right-minus-left defects in the five inequalities of
\cite[(4.10)]{RV05}.  Moreover
\[
 D_4=D_2+\bigl(J+K-M_0-D\bigr),\qquad
 D_5=D_1+\bigl(L+M_0-J-A\bigr).
\]
Each parenthesized term is at most $1$.  Thus $D_4$ or $D_5$ can equal $2$
precisely when one of the paired conditions in \cite[(4.11)]{RV05}
holds.  In the notation of that paper,
\begin{equation}\label{q:eq:e-Omega}
  e(w)=\mathbf1_{\Omega}(\{w\})+\mathbf1_{\Omega'}(\{w\}).
\end{equation}
The conclusion is therefore exactly \cite[Lemma 4.1]{RV05}, which is
obtained from the five factorial transformation identities at the end of its
Section~3.
\end{proof}

The use of three shifted tuples requires one small stability argument.

\begin{lemma}[Bounded-shift stability]\label{q:lem:shift-stability}
Let
\begin{equation}\label{q:eq:En}
 E_n=\prod_{v\in\{10,12,13\}}(vn+1)(vn+2).
\end{equation}
If $p\nmid E_n$, then for each $v\in\{10,12,13\}$ and
$r\in\{0,1,2\}$,
\begin{equation}\label{q:eq:floor-stability}
 \left\lfloor\frac{vn+r}{p}\right\rfloor
 =\left\lfloor\frac{vn}{p}\right\rfloor.
\end{equation}
Consequently, for $p^2>13n+2$ and $p\nmid E_n$, the factor
$p^{e(n/p)}$ divides both integer coefficients in \eqref{q:eq:predivisor}
for all three shifts $r=0,1,2$.
\end{lemma}

\begin{proof}
The equality \eqref{q:eq:floor-stability} can fail only if a multiple of $p$
lies among $vn+1,vn+2$.  This proves the first assertion.  In the nine
factorial arguments
\[
 10n+r,8n,8n,5n,10n,7n,12n+r,13n+r,6n
\]
only the three displayed $r$-dependent terms change.  Thus the factorial
valuation proof of Lemma \ref{q:lem:rv-divisor} applies identically to all
three shifted tuples.  Enlarging the l.c.m. factors to those in
\eqref{q:eq:predivisor} can only add divisibility.
\end{proof}

A direct breakpoint calculation from \eqref{q:eq:e-def} gives the following
nonzero intervals of $e$ on $[0,1)$.  Endpoints are left-closed and
right-open.
\begin{table}[ht]
\centering
\caption{The nonzero intervals of the periodic divisor exponent $e$.}
\label{q:tab:eintervals}
\small
\begin{tabular}{@{}ccc@{}}
\toprule
left endpoint & right endpoint & $e$\\
\midrule
$1/10$ & $1/8$ & 1\\
$1/8$ & $1/7$ & 2\\
$1/7$ & $1/6$ & 1\\
$1/5$ & $2/7$ & 1\\
$3/10$ & $1/3$ & 1\\
$3/8$ & $5/13$ & 1\\
$2/5$ & $5/12$ & 2\\
$5/12$ & $3/7$ & 1\\
$1/2$ & $7/13$ & 1\\
$3/5$ & $2/3$ & 1\\
$7/10$ & $5/7$ & 1\\
$3/4$ & $10/13$ & 1\\
$4/5$ & $5/6$ & 2\\
$5/6$ & $11/13$ & 1\\
$9/10$ & $12/13$ & 1\\
\bottomrule
\end{tabular}
\end{table}

We deliberately retain only three periods of this correction.  Define
\begin{equation}\label{q:eq:Deltan}
 \Delta_n=
 \prod_{\substack{p\ \mathrm{prime},\ p^2>13n+2,\ p\nmid E_n,\\
                   0\le n/p<3}}
 p^{e(n/p)}.
\end{equation}
By Lemma \ref{q:lem:shift-stability}, division by $\Delta_n$ is valid for all
three forms.  Thus
\begin{equation}\label{q:eq:Lambda}
 \Lambda_n=\frac{d_{13n+2}d_{12n+2}}{\Delta_n}4^n5^{3n}
\end{equation}
and
\begin{equation}\label{q:eq:integer-forms}
 L_{n,r}:=\Lambda_nF_{n,r}\in\Z+\Z\theta,
 \qquad r=0,1,2.
\end{equation}

\begin{lemma}[Prime-density saving]\label{q:lem:prime-density}
One has
\begin{equation}\label{q:eq:C3limit}
 \lim_{n\to\infty}\frac1n\log\Delta_n=\mathcal C_3,
\end{equation}
where
\begin{align}
 \mathcal C_3
 &=\sum_{[a,b),\,e}\;e\sum_{q=0}^{2}
   \left(\frac1{q+a}-\frac1{q+b}\right)\label{q:eq:C3sum}\\
 &=\frac{1947661908206663}{242860520775600}
  =8.0196727816\ldots>8.\label{q:eq:C3value}
\end{align}
Here the first sum is over the intervals in Table \ref{q:tab:eintervals}.
\end{lemma}

\begin{proof}
Fix one interval $[a,b)$ from Table \ref{q:tab:eintervals} and
$q\in\{0,1,2\}$.  The condition
\[
 q+a\le\frac np<q+b
\]
is equivalent to
\[
 \frac{n}{q+b}<p\le\frac{n}{q+a}.
\]
The prime number theorem in the form
$\vartheta(X)=\sum_{p\le X}\log p\sim X$ therefore gives the logarithmic
contribution
\[
 e\left(\frac1{q+a}-\frac1{q+b}\right)n+o(n).
\]
For these three periods the primes are bounded below by a positive multiple
of $n$, so the condition $p^2>13n+2$ is eventually automatic.  Excluding
primes dividing $E_n$ changes $\log\Delta_n$ by at most
\[
 2\sum_{p\mid E_n}\log p\le2\log E_n=O(\log n),
\]
since $e\le2$.  Summing proves \eqref{q:eq:C3limit}--\eqref{q:eq:C3sum}.
Exact rational addition gives \eqref{q:eq:C3value}.  For reference, the three
period contributions are
\[
 \frac{106501}{13860},\qquad
 \frac{373975549}{1487285800},\qquad
 \frac{44971243174667}{534293145706320}.
\]
\end{proof}

Since $\log d_N\sim N$, Lemma \ref{q:lem:prime-density} implies
\begin{equation}\label{q:eq:Lambda-root}
 \lim_{n\to\infty}\Lambda_n^{1/n}
 =4\cdot5^3\exp(25-\mathcal C_3)
 <500e^{17}.
\end{equation}

\subsection{Nonvanishing of three consecutive moments}\label{q:sec:nonvanishing}

It remains to rule out the simultaneous vanishing of the three shifted
forms.  This is done without any asymptotic saddle analysis.

\begin{lemma}[Differential equation]\label{q:lem:ode}
The function $g_n$ satisfies
\begin{equation}\label{q:eq:g-ode}
\begin{split}
 -x^2(x+4)g_n''
 &+x\bigl(72n-4+(5n-3)x\bigr)g_n'\\
 &-\bigl(320n^2+(2n-1)(3n-1)x\bigr)g_n=0.
\end{split}
\end{equation}
\end{lemma}

\begin{proof}
As noted in \eqref{q:eq:hypergeom-inner}, before subtracting the residue the
inner integral is a constant multiple of
\[
 {}_2F_1\left(6n+1,5n+1;13n+2;1-\frac{x}{z}\right).
\]
The Gauss differential equation, written in the variable $x$, is
\begin{equation}\label{q:eq:J-ode-general}
 x(z-x)J''+\bigl(z(1-2n)-(11n+3)x\bigr)J'
 -(6n+1)(5n+1)J=0.
\end{equation}
Near $x=0$, a local solution has the form $R(x)+S(x)\log x$; substituting
this into the homogeneous equation shows that its logarithmic coefficient
$S$ is itself a solution.  The residue integral in \eqref{q:eq:Jdef} is exactly
this logarithmic coefficient, so the subtraction of the constant $\Log z$
times the residue preserves \eqref{q:eq:J-ode-general}.  Finally put $z=-4$
and conjugate the operator by $g_n=x^{8n}J_n$.  A direct simplification gives
\eqref{q:eq:g-ode}.
\end{proof}

\begin{lemma}[Moment recurrence]\label{q:lem:recurrence}
For every $H\ge2$,
\begin{equation}\label{q:eq:recurrence}
\begin{split}
0={}&-5H(H-1)\mu_{H-2}
 +H(11H+77n+4)\mu_{H-1}\\
&-\bigl(7H^2+82Hn+11H+326n^2+77n+5\bigr)\mu_H\\
&+(H+2n+1)(H+3n+1)\mu_{H+1}.
\end{split}
\end{equation}
\end{lemma}

\begin{proof}
Multiply \eqref{q:eq:g-ode} by $(1-x)^H$ and integrate over $0<x<1$.
Integrate the $g_n''$ term twice and the $g_n'$ term once.  At $x=0$ all
boundary terms vanish by \eqref{q:eq:g-zero}; at $x=1$ they vanish for
$H\ge2$, since $g_n$ and $g_n'$ are regular there.  Collecting the four
resulting moments gives \eqref{q:eq:recurrence}.
\end{proof}

\begin{proposition}[Three-shift nonvanishing]\label{q:prop:nonzero}
For every $n\ge1$,
\begin{equation}\label{q:eq:triple-nonzero}
  (F_{n,0},F_{n,1},F_{n,2})\ne(0,0,0).
\end{equation}
\end{proposition}

\begin{proof}
The coefficient of $\mu_{H+1}$ in \eqref{q:eq:recurrence} is nonzero for every
$H\ge0$.  Thus if three consecutive moments
$\mu_h,\mu_{h+1},\mu_{h+2}$ vanished, applying \eqref{q:eq:recurrence} at
$H=h+2,h+3,\ldots$ would force every subsequent moment to vanish.

This is impossible.  With $n$ fixed, Proposition \ref{q:prop:F-I} gives
\[
 g_n(x)=c_nx^{8n}+O(x^{8n+1}),\qquad c_n\ne0.
\]
Hence, by the beta integral,
\begin{align*}
 \mu_H
 &=c_n\int_0^1x^{8n}(1-x)^H\dd x
   +O\!\left(\int_0^1x^{8n+1}(1-x)^H\dd x\right)\\
 &=c_n(8n)!\,H^{-(8n+1)}(1+O(H^{-1}))
 \qquad(H\to\infty).
\end{align*}
In particular the tail of the moment sequence is nonzero.  Taking $h=10n$
and recalling \eqref{q:eq:g-mu-F} proves \eqref{q:eq:triple-nonzero}.
\end{proof}

\subsection{Completion of the proof}\label{q:sec:completion}

Combining \eqref{q:eq:Fbound}, \eqref{q:eq:integer-forms} and
\eqref{q:eq:Lambda-root}, we have
\begin{equation}\label{q:eq:root-bound-pre}
 \limsup_{n\to\infty}
 \max_{0\le r\le2}|L_{n,r}|^{1/n}
 \le500e^{17}\mathcal B.
\end{equation}
The remaining numerical estimate can be made entirely rational.  The
standard series for $e$ gives
\[
 e<\sum_{k=0}^{6}\frac1{k!}
   +\frac1{7!}\frac1{1-1/8}
 =\frac{31967}{11760}
 <\frac{2719}{1000}.
\]
By Lemma \ref{q:lem:B-rational},
\begin{equation}\label{q:eq:final-exact}
 500e^{17}\mathcal B
 <500\left(\frac{2719}{1000}\right)^{17}
       \frac7{10^{11}}
 <\frac{17}{20}<1,
\end{equation}
where the last inequality is a direct integer cross-multiplication.
Therefore
\begin{equation}\label{q:eq:Ltozero}
  \max_{0\le r\le2}|L_{n,r}|\longrightarrow0.
\end{equation}

\begin{proof}[Proof of the $-1/4$ assertion in Theorem~\ref{thm:main}]
Suppose that $\theta=p/q\in\Q$ with $p\in\Z$ and $q\ge1$.  By
\eqref{q:eq:integer-forms}, each $qL_{n,r}$ is an integer.  Proposition
\ref{q:prop:nonzero} shows that for every $n$ at least one of the three numbers
$L_{n,r}$ is nonzero.  On the other hand, \eqref{q:eq:Ltozero} implies that for
all sufficiently large $n$ all three have absolute value $<1/q$.  Hence one
of the $qL_{n,r}$ is a nonzero integer of absolute value $<1$, a contradiction.
\end{proof}

\begin{remark}
The use of three shifts has only one purpose: it removes the need to prove
that a single oscillatory integral is nonzero for every $n$.  The arithmetic
cost of the shifts is $o(n)$, because the only potentially changed factorial
floors are excluded by the polynomial $E_n$ in \eqref{q:eq:En}.
\end{remark}

\section{The RV--VZ crossbreed at $1/5$}\label{sec:fifth}
The arithmetic step in this section is the first genuinely new construction in the progression: the Rhin--Viola divisor is applied term by term inside two binomial expansions of the same six-parameter Viola--Zudilin form.
\subsection{The linear forms and their normalization}\label{f:sec:construction}
Put $d_m=\lcm(1,\ldots,m)$ for $m\ge1$, and $d_0=1$.
We use the notation $J_z(h,j,k,l,m,q)$ of Viola--Zudilin~\cite{VZ18}, with the normalization in their Section~2. All references to equation and lemma numbers in~\cite{VZ18} use the MPIM preprint version. For fixed $n\geq1$ and integer $H\geq n$, define
\[
 F_n(H)=J_5(H,6n,4n,4n,5n,3n),\qquad
 J_{n,r}=F_n(6n+r)\quad(0\leq r\leq5).
\]
Thus $F_n(H)=F_n^{(0)}(H)-(\log5)F_n^{(1)}(H)$, and
\begin{align}\label{f:eq:real-original}
 F_n^{(0)}(H)
 &=5^{-3n}\int_0^1\!\int_0^1
 \frac{x^{6n}(1-x)^H y^{4n}(1-y)^{4n}(1+4y)^{4n}}
 {[x(1-y)+5y]^{5n+1}}\dd y\dd x.
\end{align}
The corresponding $F_n^{(1)}(H)$ replaces the inner integral by the positively oriented residue contour about $y=x/(x-5)$ and includes the factor $1/(2\pi i)$.

For later use, the substitution $y=\eta/(\eta-5)$ gives
\begin{align}\label{f:eq:transformed-zero}
 F_n^{(0)}(H)
 &=(-1)^{11n}\int_0^1 x^{6n}(1-x)^H
 \int_0^{-\infty}\phi_n(x,\eta)\dd\eta\dd x,\\
 \phi_n(x,\eta)
 &=\frac{\eta^{4n}(1-\eta)^{4n}}
 {(x-\eta)^{5n+1}(\eta-5)^{7n+1}}.\label{f:eq:phi}
\end{align}
If $\Gamma$ encloses $[0,1]$ and excludes $5$, then
\begin{equation}\label{f:eq:transformed-one}
 F_n^{(1)}(H)=\frac{(-1)^{11n}}{2\pi i}
 \int_\Gamma\int_0^1 x^{6n}(1-x)^H\phi_n(x,\eta)\dd x\dd\eta.
\end{equation}
Contours through $0,1$ are interpreted by limits. The endpoint convergence follows either from the elementary estimates below or from the argument for (2.11) in~\cite{VZ18}.

There are three ingredients:
\begin{equation}\label{f:eq:three}
 \begin{gathered}
 K_nJ_{n,r}\in\Z\theta+\Z\quad(0\leq r\leq5),\\
 |J_{n,r}|<6(150000000)^{-n},\\
 (J_{n,0},\ldots,J_{n,5})\ne(0,\ldots,0).
 \end{gathered}
\end{equation}
The limiting value of $\log K_n/n$ in the first ingredient is strictly below $\log(150000000)$.

\subsection{The inherited arithmetic divisor}\label{f:sec:arithmetic}
\subsubsection{Definition of the common multiplier}
Define $\mu$ on $[0,1)$ to be $1$ on
\[
 (1/6,2/7),\ (1/3,3/7),\ (1/2,4/7),\ (2/3,5/7),\ (3/4,5/6),
\]
$2$ on $(5/6,6/7)$, and $0$ elsewhere. Values at the finitely many displayed endpoints will be immaterial; set them to zero.
For $n\geq1000$ let
\begin{align}
 E_n&=\prod_{d=1}^{14}\prod_{b=-6}^{6}(dn+b),\label{f:eq:exceptional}\\
 \Delta_n&=\prod_{\substack{p\ \mathrm{prime},\ p>n/2\\p\nmid E_n}}
 p^{\mu(\{n/p\})},\label{f:eq:delta}\\
 D_n&=d_{8n+5}d_{7n+5}5^{3n}4^{2n},\qquad
 K_n=D_n/\Delta_n.\label{f:eq:K}
\end{align}
The product in \eqref{f:eq:delta} is finite: a nonzero exponent requires $n/p>1/6$, so $p<6n$. In particular, $\Delta_n\mid D_n$ and $K_n$ is a positive integer.

\begin{lemma}\label{f:lem:arith}
For $n\geq1000$ and $0\leq r\leq5$,
\[
 K_nJ_{n,r}\in\Z\theta+\Z.
\]
\end{lemma}

\subsubsection{Published inputs and the two expansions}
The direct integrality statement in~\cite[Lemma 2.1]{VZ18} gives
\[
 d_{8n+r}d_{7n+r}5^{3n}4^{\max(0,2n-r)}J_{n,r}\in\Z\theta+\Z.
\]
Consequently $D_n$ clears all coefficients for every shift under consideration.

For clarity, the mixed five-parameter integral used below is
\[
 I_z(t)=I_z^{(0)}(t)-(\log z)I_z^{(1)}(t),\qquad t=(h,j,k,l,m),
\]
where
\begin{equation}\label{f:eq:RV-definition}
 I_z^{(0)}(t)=z^{-l-m}\int_0^1\!\int_0^1
 \frac{x^j(1-x)^h y^k(1-y)^l}
 {[x(1-y)+zy]^{j+k-m+1}}\dd y\dd x.
\end{equation}
The component $I_z^{(1)}$ replaces the inner real integral by the
positively oriented residue contour about $y=x/(x-z)$, with the
factor $1/(2\pi i)$, exactly as for $J_z$.

The additional saving comes from the first binomial expansion in the
proof of~\cite[Lemma 2.1]{VZ18}, together with the involution
\[
 J_z(h,j,k,l,m,q)=J_z(j,h,m,l,k,q)
\]
of~\cite[Lemma 3.1]{VZ18}. Define
\begin{align}
 t_A(\lambda)&=(6n+r,6n,4n+\lambda,8n-\lambda,5n+\lambda),\notag\\
 t_B(\lambda)&=(6n,6n+r,5n+\lambda,9n+r-\lambda,4n+\lambda).
 \label{f:eq:expansion-tuples}
\end{align}
The two complete identities are
\begin{align}
 J_{n,r}&=\sum_{\lambda=0}^{4n}\binom{4n}{\lambda}
     5^{10n+\lambda} I_5(t_A(\lambda)),\label{f:eq:expansion-A}\\
 J_{n,r}&=\sum_{\lambda=0}^{5n+r}\binom{5n+r}{\lambda}
     5^{11n+r+\lambda} I_5(t_B(\lambda)).\label{f:eq:expansion-B}
\end{align}
The powers of $5$ are units at every prime used in $\Delta_n$.

For a five-tuple $t=(h,j,k,l,m)$ write
\[
 u_1=l+m-j,\quad u_2=m+h-k,\quad u_3=h+j-l,\quad u_4=j+k-m,
\]
\[
 \mathcal H(t)=\max(u_1,u_2,u_3,u_4),\qquad
 \mathcal K(t)=\max(u_1,\min(u_2,u_3),u_4).
\]
Theorem~2.1 of~\cite{RV05} says that, away from $2,5$, the denominator of the coefficient vector of $I_5(t)$ divides $d_{\mathcal H(t)}d_{\mathcal K(t)}$.
The two transformations
\[
 (h,j,k,l,m)\mapsto(h,m,l,k,j),\qquad
 (h,j,k,l,m)\mapsto(m+h-k,j+k-m,m,l,k)
\]
generate the Rhin--Viola orbit $\mathcal O(t)$, of size at most $12$. The quotient
\[
 \frac{I_5(t)}{h!j!k!l!m!}
\]
is unchanged by them~\cite[Section 3]{RV05}. All tuples and all four derived quantities in these orbits are nonnegative for the ranges used here. The checker verifies this at both endpoints of the affine $\lambda$ intervals.

\subsubsection{Coefficient vectors and local denominator transfer}
We use coefficient identities before evaluating at $z=5$. In the
normalization of~\cite[Theorem 2.1]{RV05}, the coefficient of
$\Li_2(1/z)$ in $I_z(t)$ is $-I_z^{(2)}(t)$, where the latter is
the companion double-contour integral. The factorial transformations
hold for this component as well. Likewise,~\cite[(2.3)]{VZ18} identifies
the corresponding coefficient of $J_z$ with $-J_z^{(2)}$.
The binomial identities hold for all these components. They therefore
give identities of rational coefficient vectors without assuming the
irrationality of the value being studied.

Write $T(t)=\prod_{i=1}^5 t_i!$, and let $\boldsymbol c(t)$ denote
the rational coefficient vector of $I_5(t)$. The factorial identities give
\begin{equation}\label{f:eq:coefficient-transfer}
 \boldsymbol c(t)=\frac{T(t)}{T(s)}\boldsymbol c(s),
 \qquad s\in\mathcal O(t).
\end{equation}
Thus, for $p\ne2,5$, a sufficient denominator exponent for $I_5(t)$ is
\[
 v_p(d_{\mathcal H(s)}d_{\mathcal K(s)})+v_p(T(s))-v_p(T(t)).
\]
Multiplication by the binomial coefficient subtracts its valuation from
this cost. The best transformation can be chosen separately for each
summand and each prime. A maximum over the summands then covers an
entire expansion; the better of two expansions bounds the same vector.
No single permutation has to realize all these minima simultaneously.

\subsubsection{The finite floor inequality}
For $p>n/2$ and $n\geq1000$, all factorial and least-common-multiple arguments in this calculation are at most $14n+2r\le14n+10<p^2$. Thus their $p$-valuations use only first powers of $p$.
For either expansion, let $a=4n$ or $5n+r$ as appropriate, and define
\begin{align}\label{f:eq:ep}
 e_p(t,\lambda)
 &=\min_{s\in\mathcal O(t)}\left(
 \ind_{\mathcal H(s)\geq p}+\ind_{\mathcal K(s)\geq p}
 +\sum_{i=1}^5\left\lfloor\frac{s_i}{p}\right\rfloor
 -\sum_{i=1}^5\left\lfloor\frac{t_i}{p}\right\rfloor\right)\notag\\
 &\hspace{8mm}-\left\lfloor\frac ap\right\rfloor
 +\left\lfloor\frac\lambda p\right\rfloor
 +\left\lfloor\frac{a-\lambda}{p}\right\rfloor.
\end{align}
The rational coefficient vector of that summand has $p$-valuation at least $-e_p(t,\lambda)$ in each coordinate. Taking the maximum over integer $\lambda$ suffices for the whole expansion; taking the better of expansions A and B is allowed.

The exact finite statement checked is
\begin{equation}\label{f:eq:floor-claim}
 \min\left\{
 \ind_{8n+r\geq p}+\ind_{7n+r\geq p},\
 \max_{\lambda\in A}e_p(t_A,\lambda),\
 \max_{\lambda\in B}e_p(t_B,\lambda)
 \right\}
 \leq\ind_{8n+5\geq p}+\ind_{7n+5\geq p}-\mu(\{n/p\}),
\end{equation}
provided $p\nmid E_n$. The following witness table suffices for both periods $0<n/p<1$ and $1<n/p<2$, and for every $r=0,\ldots,5$. Outside the listed intervals the direct denominator suffices.
\begin{center}
\begin{tabular}{ccc@{\qquad}ccc}
\toprule
$\{n/p\}$ & $\mu$ & witness & $\{n/p\}$ & $\mu$ & witness\\
\midrule
$(1/6,1/5)$ & 1&A & $(1/2,4/7)$&1&A\\
$(1/5,1/4)$ & 1&B & $(2/3,5/7)$&1&A\\
$(1/4,2/7)$ & 1&A & $(3/4,4/5)$&1&A\\
$(1/3,2/5)$ & 1&A & $(4/5,5/6)$&1&B\\
$(2/5,3/7)$ & 1&B & $(5/6,6/7)$&2&B\\
\bottomrule
\end{tabular}
\end{center}

\subsubsection{Uniformity over the bounded shifts}
The following lemma explains why a finite floor calculation proves
\eqref{f:eq:floor-claim} for every admissible $n,p$, rather than merely
for sampled values of $n$.

\begin{lemma}[Stability of the floor events]\label{f:lem:floor-stability}
Consider the affine arguments arising from the two expansions, their
Rhin--Viola orbits, the derived quantities $u_i$, and the integer
$\lambda$-range endpoints. For $n\ge1000$, $0\le r\le5$ and a prime
$p>n/2$ with $p\nmid E_n$, the denominator costs at every relevant
integer floor event and its neighbours are determined by the open
cell containing $x=n/p$ and the signed bounded shift. The cell boundaries
are among
\begin{equation}\label{f:eq:floor-cuts}
 x=i/d,\qquad
 d\in\mathcal D:=\{1,2,3,4,5,6,7,8,9,10,12,13,14\},
 \qquad 0\le i\le2d.
\end{equation}
There are $108$ open cells in $0<x<2$.
\end{lemma}
\begin{proof}
Every relevant argument has the form
\[
 f(\lambda)=An+B\lambda+Cr,\qquad B\in\{-1,0,1\}.
\]
When $B\ne0$, its events occur at
\[
 \lambda_0=\frac{kp-An-Cr}{B}.
\]
These are integers because $B=\pm1$. The cost is constant on integer
blocks between events. Testing the range endpoints and
$\lambda_0-1,\lambda_0,\lambda_0+1$ therefore covers every possible
value, including both directions of a floor jump. The indicator
conditions involving $\mathcal H$ and $\mathcal K$ are Boolean
combinations of the tests $u_i\ge p$, so they introduce no other events.

At a neighbour of an event, a second affine argument
$g(\lambda)=Dn+E\lambda+Fr$ becomes
\begin{equation}\label{f:eq:event-substitution}
 g(\lambda_0+\delta)=
 \left(D-\frac{EA}{B}\right)n+\frac EB kp
 +\left(F-\frac{EC}{B}\right)r+E\delta,
 \qquad \delta\in\{-1,0,1\}.
\end{equation}
The coefficient of $kp$ is integral. Exact enumeration of the affine
coefficient lists gives
\[
 \left|D-\frac{EA}{B}\right|\le14,
 \qquad \left|F-\frac{EC}{B}\right|\le1,
 \qquad |E|\le1.
\]
These determinant bounds are checked in the arithmetic verifier, as are
the affine endpoint inequalities ensuring admissibility of every orbit
tuple. Thus every remaining comparison is between $dn+e$ and a
multiple of $p$, with
\[
 |d|\le14,\qquad |e|\le r+1\le6.
\]
The same calculation applies to the membership tests
$\lambda\ge0$ and $a-\lambda\ge0$.

The definition of $E_n$ implies the integer separation
\begin{equation}\label{f:eq:prime-separation}
 p\nmid E_n\quad\Longrightarrow\quad
 \dist(dn,p\Z)\ge7\qquad(1\le |d|\le14).
\end{equation}
Indeed, a distance at most $6$ would give a factor $dn+b$ divisible by
$p$, after changing both signs when necessary. Hence no perturbation
$|e|\le6$ crosses an unshifted event when $d\ne0$.
When $d=0$, the sign of $e$ determines the comparison, with $e=0$
retained as equality. Here $p>500>|e|$. This handles coincident events
that separate under a shift, including the short integer intervals
between them.

All the nonzero $n$-coefficients of these comparisons have absolute
value in $\mathcal D$, as verified by enumerating the affine lists.
They give precisely the subdivision~\eqref{f:eq:floor-cuts}.
A boundary value of $n/p$ would imply $p\mid dn$, and is excluded by
$p\nmid E_n$. On each remaining cell, the unshifted comparisons are
fixed; the signed bounded comparisons just described are fixed too.
This proves the assertion.
\end{proof}

\begin{proof}[Finite verification of \eqref{f:eq:floor-claim} and completion of Lemma~\ref{f:lem:arith}]
For each open cell take its reduced rational midpoint $N/D$, and use the
integer representative $(n_*,p_*)=(1000N,1000D)$.
This is a floor calculation: $p_*$ need not be prime. Every nonzero
distance to an unshifted integer boundary is at least $1000$, so it
realizes the stable comparisons in Lemma~\ref{f:lem:floor-stability}.
Coincident events retain the signed shift and integer neighbour.

The program \texttt{verify\_arithmetic.py} constructs the affine lists
and orbit transformations, generates the subdivision, and tests the
prescribed A/B witness from the displayed table for all six shifts.
It checks $300$ nontrivial cell/shift witnesses and $23940$ deduplicated
integer-$\lambda$ event tests, using only integers and exact rational
numbers. It does not read a table of purported answers. When $\mu=0$,
the direct VZ denominator suffices without further tests.
Lemma~\ref{f:lem:floor-stability} transfers this finite calculation to
all the $n,p,r$ under consideration, proving~\eqref{f:eq:floor-claim}.
Since $D_n$ already clears every coefficient, the local inequality at
each prime in $\Delta_n$ proves Lemma~\ref{f:lem:arith}.
\end{proof}

\begin{remark}[Why the exceptional primes are removed]
This exclusion is needed for the stated sufficient denominator bound.
For example, at $n=1000$, $p=4001$ and $r=2$, the target cost is $1$,
but both expansions have worst cost $2$; in expansion B this occurs at
$\lambda=1001$. The prime divides $4n+1$, hence is excluded by $E_n$.
This is a failure of the two-expansion bound without the exclusion, not
a claim that the actual coefficients fail to be integral. Its
asymptotic cost is only logarithmic, as we now show.
\end{remark}

\subsubsection{Asymptotic cost}
The prime number theorem gives
\begin{equation}\label{f:eq:saving}
 \lim_{n\to\infty}\frac{\log\Delta_n}{n}
 =\int_0^2\frac{\mu(\{x\})}{x^2}\dd x
 =\frac{708641}{180180}=:S.
\end{equation}
Indeed, each interval gives a difference of reciprocal endpoints. Omitting primes dividing $E_n$ changes $\log\Delta_n$ by at most $2\log E_n=O(\log n)$, since $E_n$ is a polynomial of degree $182$ and $\log E_n=182\log n+O(1)$.
Also $\log d_{cn+O(1)}/n\to c$. Consequently
\begin{equation}\label{f:eq:cost}
 C:=\lim_{n\to\infty}\frac{\log K_n}{n}
 =15+3\log5+2\log4-S
 =18.66794097658059939998\ldots.
\end{equation}
The decimal is not used to establish a sign.

\subsection{A fixed-contour bound with rational certificates}\label{f:sec:contours}
Set
\[
 R=\frac1{150000000},\qquad a=\frac45,\quad b=\frac74,
 \qquad X(t)=t-ia\,t(1-t),\quad Y(u)=u+ib\,u(1-u).
\]
The two phase functions to be bounded are
\begin{align}
 f_0(t,u)&=\frac{[t(1-t)]^6[u(1-u)]^4(1+4u)^4}
 {125[t(1-u)+5u]^5},\label{f:eq:f0}\\
 f_1(x,y)&=\frac{x^6(1-x)^6y^4(1-y)^4}{(x-y)^5(y-5)^7}.\label{f:eq:f1}
\end{align}
At the removable corner limits they are set to zero.

\begin{lemma}\label{f:lem:phase}
For $0\leq t,u\leq1$,
\[
 0\leq f_0(t,u)<R,\qquad |f_1(X(t),Y(u))|<R.
\]
\end{lemma}

\begin{proof}[Exact certificate and its analytic justification]
Write $T=t(1-t)$, $U=u(1-u)$ and
\[
 S_c(v)=v^2\{1+c^2(1-2v)+c^4v^2\}.
\]
Then
\begin{equation}\label{f:eq:modulus}
 |f_1(X(t),Y(u))|^2
 =\frac{S_a(T)^6S_b(U)^4}
 {[(t-u)^2+(aT+bU)^2]^5[(5-u)^2+b^2U^2]^7}.
\end{equation}
The function $S_c$ is increasing on $[0,1/4]$: its derivative is
\[
 2v\{1+c^2-3c^2v+2c^4v^2\}\geq2v(1+c^2/4)\geq0.
\]
Also $(5-u)^2+b^2u^2(1-u)^2$ is decreasing on $[0,1]$, since its derivative is at most $-8+b^2/2<0$.

For a rational box $[t_-,t_+]\times[u_-,u_+]$, compute the exact minima and maxima $T_-,T_+,U_-,U_+$ of $t(1-t),u(1-u)$. Put
\[
 d=\max(0,t_--u_+,u_--t_+).
\]
An upper bound for \eqref{f:eq:modulus} is
\begin{equation}\label{f:eq:complex-box}
 \frac{S_a(T_+)^6S_b(U_+)^4}
 {[d^2+(aT_-+bU_-)^2]^5[(5-u_+)^2+b^2u_+^2(1-u_+)^2]^7}.
\end{equation}
For the real phase an upper bound is
\begin{equation}\label{f:eq:real-box}
 \frac{T_+^6[U_+(1+4u_+)]^4}{125[t_-(1-u_-)+5u_-]^5}.
\end{equation}
These are rational expressions.

The corner patches are handled separately, with $\varepsilon=1/16$. On $0\leq t,u\leq\varepsilon$ one has
\[
 \frac{t^6u^4}{(t+u)^5}\leq\frac{\varepsilon^5}{32}.
\]
To see this, first increase $t$ to $\varepsilon$; the resulting expression is increasing in $u\leq\varepsilon$. This proves the bound by evaluating at the top corner.
For the complex phase this gives, at both equal-endpoint corners,
\begin{equation}\label{f:eq:corner-complex}
 |f_1|^2\leq
 \frac{\varepsilon^{10}(1+a^2)^6(1+b^2)^4}
 {1024a^{10}(1-\varepsilon)^{10}4^{14}}<R^2.
\end{equation}
Here $b\geq a$, $|Y-5|\geq4$, and $S_c(v)\leq v^2(1+c^2)$ for $c=a,b$ and $0\leq v\leq1/4$.
For the real phase, at $(0,0)$,
\begin{equation}\label{f:eq:corner-real}
 f_0\leq\frac{\varepsilon^5(1+4\varepsilon)^4}{4000(1-\varepsilon)^5}<R.
\end{equation}
Both final comparisons are exact rational comparisons.

The remaining unit square is covered by $458$ rational boxes for the complex phase and $205$ for the real phase. Every bound \eqref{f:eq:complex-box} is strictly less than $R^2$, and every bound \eqref{f:eq:real-box} is strictly less than $R$. The certificate files contain the box endpoints. The independent checker \texttt{verify\_contours.py} checks every inequality by cross multiplication and verifies a disjoint cover of the square, including the corner patches. The cover check is stronger than merely adding box areas: on every vertical slab induced by the endpoints, it checks that the vertical intervals tile $[0,1]$ with no gap or overlap.
\end{proof}

\begin{lemma}\label{f:lem:integral-bound}
For every $n\geq1$ and $0\leq r\leq5$,
\[
 |J_{n,r}|<6R^n.
\]
\end{lemma}
\begin{proof}
In \eqref{f:eq:real-original}, with $H=6n+r$, the integrand is $f_0(x,y)^n(1-x)^r/[x(1-y)+5y]$. Since
\[
 x(1-y)+5y\geq x+y,\qquad
 \int_0^1\!\int_0^1\frac{\dd x\dd y}{x+y}=2\log2<2,
\]
we obtain $|J_{n,r}^{(0)}|<2R^n$.

For \eqref{f:eq:transformed-one}, take the upper and lower parabolic
arcs $Y([0,1])$ and its conjugate, oriented to surround $[0,1]$.
For an upper-half-plane $\eta$, deform the inner $x$ path to
$X([0,1])$ in the lower half-plane. No pole is crossed. Use the
conjugate $x$ path for the lower arc.

Here is a uniform endpoint estimate for this deformation. Along the
straight homotopy put
\[
 X_s(t)=t-is\frac45t(1-t),\qquad 0\le s\le1.
\]
For $0\le t,u\le1/2$, writing $D=|X_s(t)-Y(u)|$, we have
\[
 D\ge|t-u|,\qquad D\ge\frac b2u,
 \qquad t+u\le|t-u|+2u\le\left(1+\frac4b\right)D.
\]
Consequently
\begin{equation}\label{f:eq:homotopy-separation}
 |X_s(t)-Y(u)|\ge\frac7{23}(t+u),
 \qquad 0\le t,u\le\tfrac12,
\end{equation}
uniformly in $s$. Near $(1,1)$ the same bound holds with $t,u$
replaced by $1-t,1-u$. For each fixed $n,r$, the full integrand
including the path derivatives is therefore
$O((t+u)^{5n-1})$ at $(0,0)$ and
$O(((1-t)+(1-u))^{5n+r-1})$ at $(1,1)$.
These bounds are locally integrable. They justify first performing the
deformation on truncated paths and then passing to the endpoint limits;
the small connecting pieces have vanishing contributions. The starting
contour representation is~\eqref{f:eq:transformed-one}, equivalently
\cite[(2.11)]{VZ18}.

We have $|1-X(t)|\leq1$, $|X'(t)Y'(u)|<3$, and $|Y(u)-5|\geq4$. Moreover
\[
 I:=\int_0^1\!\int_0^1\frac{\dd t\dd u}{|X(t)-Y(u)|}<7.
\]
Indeed, on $[0,1/2]^2$ the denominator is at least $a(t+u)/2$; the analogous bound holds on $[1/2,1]^2$. These two contributions are at most $5\log2$. On either off-diagonal half-square, the denominator is at least $1/4$. To prove the latter assertion, either $|t-u|\geq1/4$, or both variables lie in $[1/4,3/4]$, when the imaginary separation is at least $a(3/8)=3/10$. The two off-diagonal contributions total at most $2$. Thus $I\leq5\log2+2<7$.

The two half-contours and their conjugacy now give
\[
 |J_{n,r}^{(1)}|\leq\frac{1}{\pi}\frac34 I R^n
 <\frac74R^n<2R^n,
\]
using $\pi>3$. Since $\log5<2$,
\[
 |J_{n,r}|\leq |J_{n,r}^{(0)}|+(\log5)|J_{n,r}^{(1)}|<6R^n.
\]
\end{proof}

\subsection{A contiguous recurrence and nonvanishing}\label{f:sec:nonvanishing}
\begin{lemma}\label{f:lem:recurrence}
For $n\geq2$ and $H\geq6n$,
\begin{equation}\label{f:eq:recurrence}
 \sum_{j=0}^5 c_j(H,n)F_n(H+j)=0,
\end{equation}
where
\begin{align*}
 c_0={}&4(H+1)(H+2)(H-n+1),\\
 c_1={}&(H+2)(-15H^2-65Hn-70H+148n^2-69n-79),\\
 c_2={}&20H^3+215H^2n+190H^2+31Hn^2+1112Hn+592H\\
       &\quad-644n^3-314n^2+1400n+608,\\
 c_3={}&-10H^3-215H^2n-130H^2-501Hn^2-1542Hn-542H\\
       &\quad+390n^3-1347n^2-2719n-734,\\
 c_4={}&(H+n+4)(65Hn+10H+252n^2+334n+44),\\
 c_5={}&(H+n+4)(H+n+5)(H+2n+5).
\end{align*}
In particular, $c_5(H,n)>0$ in this range.
\end{lemma}
\begin{proof}[Certificate and boundary terms]
Let
\[
 w=x^{6n}(1-x)^H\phi_n(x,y),\qquad
 p_0=x(1-x)(x-y),\quad q_0=y(1-y)(y-5)(x-y).
\]
The certificate file \texttt{contiguous\_certificate.json} supplies $P,Q$ and coefficients $b_0,\ldots,b_6$, with $b_0=0$, satisfying the polynomial identity
\begin{equation}\label{f:eq:telescoper}
 \partial_x(p_0Pw)+\partial_y(q_0Qw)
 =w\sum_{j=0}^6 b_j(H,n)(1-x)^j.
\end{equation}
After cancelling $w$, its verification is a polynomial calculation over $\Z[x,y,H,n]$. The checker \texttt{verify\_recurrence.py} performs it directly, without rerunning the nullspace search used to find $P,Q$.

The raw identity can be integrated for $n\ge2$ and every integer
$H\ge0$, using the explicit integral representations. For the real
integral~\eqref{f:eq:transformed-zero}, the $x$ boundary terms vanish
because of $x^{6n+1}(1-x)^{H+1}$, and the $y=0$ term also vanishes.
At the joint corner $x=y=0$, both fluxes have an additional quadratic
vanishing factor relative to the weight, and are
$O((x+|y|)^{5n+1})$. The finite endpoint contributions thus vanish.

Write $U=p_0Pw$ and $V=q_0Qw$. The certificate satisfies
$\deg_yP\le4$ and $\deg_yQ\le2$, so as $y\to-\infty$,
\[
 U=O(|y|^{3-4n}),\qquad V=O(|y|^{4-4n}).
\]
The first flux is integrable and the second tends to zero for $n\ge2$.
For the contour integral, use a fixed closed $y$ contour away from
the endpoints. The $y$ flux is a derivative of a single-valued rational
function and its integral is zero; the $x$ boundary terms vanish as
before. Thus~\eqref{f:eq:telescoper} integrates to zero for both
$F_n^{(0)}$ and $F_n^{(1)}$, and hence for $F_n$.

The common factor in $b_1,\ldots,b_6$ is $5n(H+n+3)(H+n+4)$. Dividing by it and replacing $H$ by $H-1$ gives exactly the displayed $c_j$. The raw identity is available at $H-1$ in the stated range, and these operations introduce no zero denominator there.
\end{proof}

\begin{proposition}[Eventual positivity for fixed $n$]\label{f:prop:positivity}
For every fixed integer $n\ge1$,
\begin{equation}\label{f:eq:eventual-positive}
 F_n(H)\sim\frac{(4n)!(n-1)!}{5^{7n+1}}\,H^{-(5n+1)}
 \qquad(H\longrightarrow\infty).
\end{equation}
In particular, $F_n(H)>0$ for all sufficiently large $H$.
\end{proposition}
\begin{proof}
Write
\[
 g_n(x)=5^{-3n}\int_0^1
 \frac{y^{4n}(1-y)^{4n}(1+4y)^{4n}}
 {[x+(5-x)y]^{5n+1}}\dd y,
\]
so that $F_n^{(0)}(H)=\int_0^1x^{6n}(1-x)^Hg_n(x)\dd x$.
The substitution $y=xt/5$ gives
\begin{equation}\label{f:eq:inner-scaling}
 x^ng_n(x)=5^{-7n-1}\int_0^{5/x}
 \frac{t^{4n}(1-xt/5)^{4n}(1+4xt/5)^{4n}}
 {[1+(1-x/5)t]^{5n+1}}\dd t.
\end{equation}
For $0<x\le1$, the numerator factors other than $t^{4n}$ are bounded
by $5^{4n}$ on this integration interval, and the denominator is at
least $(1+4t/5)^{5n+1}$. After extending the integrand by zero outside
the interval, an integrable majorant on $[0,\infty)$ is therefore a
constant depending on $n$ times
\[
 \frac{t^{4n}}{(1+4t/5)^{5n+1}}.
\]
Its tail is $O(t^{-n-1})$. Dominated convergence yields
\begin{equation}\label{f:eq:inner-limit}
 \lim_{x\downarrow0}x^ng_n(x)
 =5^{-7n-1}\int_0^\infty\frac{t^{4n}}{(1+t)^{5n+1}}\dd t
 =\frac{(4n)!(n-1)!}{5^{7n+1}(5n)!}.
\end{equation}

Set $b_n(x)=x^ng_n(x)$. The same majorant shows that $b_n$ is
bounded on $(0,1]$. Substituting $x=t/H$ gives
\[
 H^{5n+1}F_n^{(0)}(H)=
 \int_0^H t^{5n}(1-t/H)^H b_n(t/H)\dd t.
\]
Since $(1-t/H)^H\le e^{-t}$, another application of dominated
convergence shows that the limit is
$(5n)!b_n(0)=(4n)!(n-1)!/5^{7n+1}$.

For the residue component, put
\begin{equation}\label{f:eq:residue-function}
 A_n(x)=\Res_{\eta=x}\phi_n(x,\eta)
 =\frac{(-1)^{5n+1}}{(5n)!}\frac{d^{5n}}{dx^{5n}}
 \left(\frac{x^{4n}(1-x)^{4n}}{(x-5)^{7n+1}}\right).
\end{equation}
It is rational with no pole on $[0,1]$, hence bounded there. The residue
formula in~\eqref{f:eq:transformed-one} gives
\[
 F_n^{(1)}(H)=(-1)^{11n}\int_0^1x^{6n}(1-x)^H A_n(x)\dd x
 =O(H^{-6n-1}).
\]
This is smaller than the real component by a factor $O(H^{-n})$.
Subtracting $(\log5)F_n^{(1)}(H)$ proves~\eqref{f:eq:eventual-positive}.
\end{proof}

\begin{lemma}\label{f:lem:nonzero}
For every $n\ge2$, five consecutive values
$F_n(H),\ldots,F_n(H+4)$ with $H\ge6n$ cannot all vanish.
In particular, at least one of $J_{n,0},\ldots,J_{n,5}$ is nonzero.
\end{lemma}
\begin{proof}
If a block of five values vanished, Lemma~\ref{f:lem:recurrence} and
$c_5(H,n)>0$ would force all subsequent values to vanish.
For that fixed $n$, this contradicts Proposition~\ref{f:prop:positivity}.
\end{proof}

\begin{remark}
The limit in Proposition~\ref{f:prop:positivity} keeps $n$ fixed;
no uniformity in $n$ is needed. It does not assert that the bounded-shift
forms themselves are positive. Five shifts suffice for the
nonvanishing argument, but we retain the six shifts already covered by
the arithmetic and contour certificates.
\end{remark}

\subsection{Closing the irrationality argument}\label{f:sec:conclusion}
Combine Lemmas~\ref{f:lem:arith}, \ref{f:lem:integral-bound} and \ref{f:lem:nonzero}. For all sufficiently large $n$, the six numbers $K_nJ_{n,r}$ have integer coefficient vectors, at least one is nonzero, and
\[
 \max_{0\leq r\leq5}|K_nJ_{n,r}|
 \leq6\exp\{-(\gamma+o(1))n\},
\]
where
\begin{align}\label{f:eq:gap}
 \gamma&=\log(150000000)-C\\
 &=\log75000-15+\frac{708641}{180180}
 =0.15820487547993045413\ldots.\notag
\end{align}
There is a short exact positivity check. The identity
\[
 \log z=2\sum_{k\geq0}\frac{1}{2k+1}\left(\frac{z-1}{z+1}\right)^{2k+1}
 \quad(z>1)
\]
has positive terms. Its first $2,3,4$ terms at $z=2,3,5$, respectively, give
\[
 \log2>\frac{56}{81}>\frac{69}{100},\quad
 \log3>\frac{263}{240}>\frac{109}{100},\quad
 \log5>\frac{122492}{76545}>\frac85.
\]
Since $75000=2^3\cdot3\cdot5^5$ and $S>393/100$, this proves
\[
 \gamma>3\frac{69}{100}+\frac{109}{100}+5\frac85-15+\frac{393}{100}
 =\frac9{100}>0.
\]
Thus the maximum of the absolute values of the six integer linear forms tends to zero. If $\theta=a/b\in\Q$, every nonzero such form would have absolute value at least $1/|b|$. This contradicts Lemma~\ref{f:lem:nonzero} and proves the $1/5$ assertion in Theorem~\ref{thm:main}.

\section{The crossbreed on the negative side at $-1/3$}\label{sec:third}
We now retain the same six-parameter ray and inherited factorial mechanism, but continue the integral to $z=-3$.  The arithmetic is therefore structurally the same as at $1/5$, while the analytic side uses a compact complex contour together with a real tail.
\subsection{Construction and target estimates}
Put
\[
 \theta=\Li_2(-1/3)=\sum_{k=1}^{\infty}\frac{(-1)^k}{3^k k^2},
 \qquad d_N=\lcm(1,\ldots,N).
\]
We use the normalization of Viola--Zudilin \cite{VZ18}. For fixed integer
$n\ge1$ and integer $H\ge0$, let
\[
 F_n(H)=J_{-3}(H,6n,4n,4n,5n,3n),\qquad
 J_{n,r}=F_n(6n+r),\quad 0\le r\le5.
\]
For $H$ outside the admissible arithmetic range, $F_n(H)$ is understood
through the continued integral below. We invoke the arithmetic theorem only
for the six displayed shifts.

For $z>1$ the original real component is
\begin{equation}\label{t:eq:original}
 J_z^{(0)}(H,6n,4n,4n,5n,3n)
 =z^{-3n}\int_0^1\!\int_0^1
 \frac{x^{6n}(1-x)^H y^{4n}(1-y)^{4n}(1-y+zy)^{4n}}
 {[x(1-y)+zy]^{5n+1}}\dd y\dd x.
\end{equation}
The mixed form is $J_z=J_z^{(0)}-(\log z)J_z^{(1)}$, with the standard
positively oriented inner residue contour defining $J_z^{(1)}$.
The real-square formula \eqref{t:eq:original} is \emph{not} used directly at
$z=-3$, where its denominator has an interior zero.

The argument establishes, for a common positive integer $K_n$,
\begin{equation}\label{t:eq:three}
 \begin{gathered}
 K_nJ_{n,r}\in\Z\theta+\Z\quad(n\ge1000,\ 0\le r\le5),\\
 |J_{n,r}|<14(26000000)^{-n}\quad(n\ge1,\ 0\le r\le5),\\
 (J_{n,0},\ldots,J_{n,5})\ne(0,\ldots,0)\quad(n\ge2),
 \end{gathered}
\end{equation}
and
\[
 \log26000000-\lim_{n\to\infty}\frac{\log K_n}{n}>
 \frac1{20}.
\]

\subsection{The negative-argument continuation}\label{t:sec:continuation}
The substitution $y=\eta/(\eta-z)$ in \eqref{t:eq:original} gives
\[
 \sigma_n=(-1)^{11n},\qquad
 \phi_n(x,\eta)=
 \frac{\eta^{4n}(1-\eta)^{4n}}
 {(x-\eta)^{5n+1}(\eta+3)^{7n+1}}.
\]
Following \cite[(2.7)--(2.11)]{VZ18} and the negative-$z$ implementation in \cite[\S2]{RV19}, choose the ray $[0,-i\infty)$ in the
$\eta$-plane and continue $z$ to $-3$ through the upper half-plane. Then
\begin{align}
 F_n^{(0)}(H)&=\sigma_n\int_0^1 x^{6n}(1-x)^H
       \int_0^{-i\infty}\phi_n(x,\eta)\dd\eta\dd x,
       \label{t:eq:vertical}\\
 F_n^{(1)}(H)&=\frac{\sigma_n}{2\pi i}\int_\Gamma\int_0^1
       x^{6n}(1-x)^H\phi_n(x,\eta)\dd x\dd\eta,
       \label{t:eq:residue}\\
 F_n(H)&=F_n^{(0)}(H)-(\log3+i\pi)F_n^{(1)}(H).
       \label{t:eq:mixed}
\end{align}
Here $\Gamma$ positively encloses $[0,1]$ and excludes $-3$. A fixed contour
away from the endpoints can always be used in \eqref{t:eq:residue}.
These formulas fix the branch and orientation. There is no additional
exponential power of $3$ in \eqref{t:eq:vertical}.

The ray $[0,-i\infty)$ avoids both poles $\eta=x$ and $\eta=-3$.
At infinity $\phi_n=O(|\eta|^{-4n-2})$, and near the joint endpoint
$x=\eta=0$ the full weight is $O((x+|\eta|)^{5n-1})$.
Thus the vertical integral converges absolutely. The published polynomial
coefficient identities continue to this branch. In particular, the forms in
the admissible range are real, despite the complex representation
\eqref{t:eq:mixed}. An explicit real representation is given in
Section~\ref{t:sec:contours}.

\subsubsection{An explicit algebraic check of the branch}
For the lower ray, the logarithmic convention is precisely
\cite[(2.10)]{VZ18}. The following identity makes its cancellation inspectable
without using the contour estimate later in the proof. A fixed residue
contour such as $|\eta-1/2|=3/4$ encloses $[0,1]$ and avoids the entire
continuation path $z=3e^{it}$, $0\le t\le\pi$.

\begin{lemma}[Rational-logarithmic branch identity]\label{t:lem:branch-identity}
For fixed $n\ge1$ and $0<x<1$, let
\[
 \rho_n(x)=\Res_{\eta=x}\phi_n(x,\eta),\qquad
 g_n(x)=\sigma_n\int_0^{-i\infty}\phi_n(x,\eta)\dd\eta.
\]
There is $R_n(x)\in\Q(x)$ such that
\begin{equation}\label{t:eq:branch-identity}
 g_n(x)=\sigma_n\bigl[R_n(x)+\rho_n(x)
       \{\log(3/x)+i\pi\}\bigr].
\end{equation}
Consequently, for every integer $H\ge0$,
\begin{equation}\label{t:eq:real-density}
 F_n(H)=\sigma_n\int_0^1x^{6n}(1-x)^H
       \{R_n(x)-\rho_n(x)\log x\}\dd x.
\end{equation}
\end{lemma}
\begin{proof}
Put $N=5n+1$, $M=7n+1$ and expand exactly
\[
 \phi_n(x,\eta)=\sum_{j=1}^{N}\frac{a_j(x)}{(\eta-x)^j}
              +\sum_{j=1}^{M}\frac{b_j(x)}{(\eta+3)^j}.
\]
There is no polynomial part. Since $\phi_n=O(\eta^{-4n-2})$ at infinity, $a_1+b_1=0$, and $a_1=\rho_n$.
The simple-pole integral is
\[
 a_1\bigl[\Log(\eta-x)-\Log(\eta+3)\bigr]_0^{-i\infty}
 =a_1\{\log(3/x)+i\pi\},
\]
because the lower-half-plane boundary value at $-x$ is $\log x-i\pi$. The remaining terms give the real rational function
\[
 R_n(x)=\sum_{j=2}^{N}\frac{a_j(x)(-x)^{1-j}}{j-1}
       +\sum_{j=2}^{M}\frac{b_j(x)3^{1-j}}{j-1}.
\]
This proves \eqref{t:eq:branch-identity}. Subtracting $(\log3+i\pi)$ times the residue integral gives \eqref{t:eq:real-density}.
\end{proof}

\subsection{The inherited arithmetic divisor}\label{t:sec:arithmetic}
\subsubsection{A common multiplier}
Define $\mu$ on $[0,1)$ by the following table and set it to zero elsewhere,
including the displayed endpoints.
\begin{center}
\begin{tabular}{cc@{\qquad}cc}
\toprule
interval & $\mu$ & interval & $\mu$\\
\midrule
$(1/6,2/7)$&1&$(2/3,5/7)$&1\\
$(1/3,3/7)$&1&$(3/4,5/6)$&1\\
$(1/2,4/7)$&1&$(5/6,6/7)$&2\\
\bottomrule
\end{tabular}
\end{center}
For $n\ge1000$ put
\begin{align}
 E_n&=\prod_{d=1}^{14}\prod_{b=-6}^{6}(dn+b),\label{t:eq:E}\\
 \Delta_n&=\prod_{\substack{p\ \mathrm{prime},\ p>n/4\\p\nmid E_n}}
       p^{\mu(\{n/p\})},\label{t:eq:Delta}\\
 D_n&=d_{8n+5}d_{7n+5}\,3^{3n}4^{2n},\qquad K_n=D_n/\Delta_n.
       \label{t:eq:K}
\end{align}
Every prime with nonzero exponent in \eqref{t:eq:Delta} is less than $6n$ and
has exponent at most two. Both l.c.m. factors in $D_n$ contain that prime,
so $\Delta_n\mid D_n$ and $K_n$ is a positive integer.

\begin{lemma}\label{t:lem:arithmetic}
For $n\ge1000$ and $0\le r\le5$, one has $K_nJ_{n,r}\in\Z\theta+\Z$.
\end{lemma}

\subsubsection{The two expansions and the factorial transfer}
Lemma~2.1 of \cite{VZ18}, continued as above, first gives
\[
 d_{8n+r}d_{7n+r}\,3^{3n}4^{\max(0,2n-r)}J_{n,r}\in\Z\theta+\Z.
\]
The precise source normalization is important. In the initial
Lemma~2.1 of \cite{VZ18}, the exponent of $1-z$ is
$\beta=\max(0,k+l-h)$, hence $\max(0,2n-r)$ here. The later symmetrized
Lemma~3.3 uses the larger exponent $\max(0,k+l-h,l+m-j)$, which would be
$3n$ on this ray. We do not use that enlargement: the opening paragraph of
its proof explicitly continues the polynomial identities (2.2)--(2.3) of
Lemma~2.1 to the full domain of holomorphy, before symmetrizing the bounds.
Thus the sharper exponent above remains valid on the negative branch, and
no additional factor $4^n$ is required.

Consequently $D_n$ clears all coefficients. To improve this denominator,
write $I_z(h,j,k,l,m)$ for the mixed Rhin--Viola integral. The first binomial
expansion in \cite[Lemma 2.1]{VZ18}, first for the original tuple and then
for its elementary involution $(h,j,k,l,m,q)\mapsto(j,h,m,l,k,q)$, gives
\begin{align}
 J_{n,r}&=\sum_{\lambda=0}^{4n}\binom{4n}{\lambda}(-3)^{10n+\lambda}
 I_{-3}(6n+r,6n,4n+\lambda,8n-\lambda,5n+\lambda),\label{t:eq:A}\\
 J_{n,r}&=\sum_{\lambda=0}^{5n+r}\binom{5n+r}{\lambda}(-3)^{11n+r+\lambda}
 I_{-3}(6n,6n+r,5n+\lambda,9n+r-\lambda,4n+\lambda).\label{t:eq:B}
\end{align}
Call these expansions A and B. The powers of $-3$ are units at the primes
used in \eqref{t:eq:Delta}.

For a five-tuple $t=(h,j,k,l,m)$ set
\[
 u_1=l+m-j,\quad u_2=h+m-k,\quad u_3=h+j-l,\quad u_4=j+k-m,
\]
\[
 \mathcal H(t)=\max(u_1,u_2,u_3,u_4),\qquad
 \mathcal K(t)=\max(u_1,\min(u_2,u_3),u_4),\qquad
 T(t)=h!j!k!l!m!.
\]
At primes different from $2,3$, Theorem~2.1 of \cite{RV05} clears the
coefficient denominators of $I_{-3}(t)$ with
$d_{\mathcal H(t)}d_{\mathcal K(t)}$. The transformations
\[
 (h,j,k,l,m)\mapsto(h,m,l,k,j),\qquad
 (h,j,k,l,m)\mapsto(h+m-k,j+k-m,m,l,k)
\]
generate an orbit $\mathcal O(t)$ of size at most twelve. Their factorial
identities give, for the rational coefficient vectors,
\begin{equation}\label{t:eq:factorial}
 \boldsymbol c(t)=\frac{T(t)}{T(s)}\boldsymbol c(s),\qquad s\in\mathcal O(t).
\end{equation}
All required tuples and complementary parameters are nonnegative throughout
the two summation ranges. This is checked at the endpoints of their affine
parameter intervals.

These are coefficient identities, not an assumption that $1$ and $\theta$
are independent. For example, the companion double-contour integral
identifies the dilogarithm coefficient before specialization, and satisfies
the same binomial and factorial identities; see \cite[(2.3)]{VZ18}.

\subsubsection{An exact floor certificate and its all-\texorpdfstring{$n$}{n} interpretation}
For $p>n/4$ and $n\ge1000$, all factorial and l.c.m. arguments below are at
most $14n+10<p^2$. Their valuations therefore require only first-floor terms.
Let $t=t(\lambda)$ be the tuple in A or B and let $a=4n$ or $5n+r$,
respectively. A sufficient denominator exponent for its summand is
\begin{align}
 e_p(t,\lambda)=\min_{s\in\mathcal O(t)}\biggl(&
 \ind_{\mathcal H(s)\ge p}+\ind_{\mathcal K(s)\ge p}
 +\sum_{i=1}^5\left\lfloor\frac{s_i}{p}\right\rfloor
 -\sum_{i=1}^5\left\lfloor\frac{t_i}{p}\right\rfloor\biggr)
 \notag\\*[-1mm]
 &-\left\lfloor\frac ap\right\rfloor
 +\left\lfloor\frac\lambda p\right\rfloor
 +\left\lfloor\frac{a-\lambda}{p}\right\rfloor.
 \label{t:eq:cost-prime}
\end{align}
The sign of the factorial quotient follows directly from
\eqref{t:eq:factorial}. The best orbit element may be chosen separately for
each prime and each summand. Taking the maximum over $\lambda$ covers the
sum; taking the better of A and B is legitimate because they represent the
same form.

\begin{lemma}[Finite-to-all-$n$ floor stability]\label{t:lem:stability}
If $n\ge1000$, $p>n/4$, $p\nmid E_n$ and $\mu(\{n/p\})>0$, then
\begin{equation}\label{t:eq:floor-result}
 \min\left\{\max_{\lambda\in A}e_p(t_A,\lambda),
             \max_{\lambda\in B}e_p(t_B,\lambda)\right\}
 \le 2-\mu(\{n/p\}).
\end{equation}
\end{lemma}
\begin{proof}
Every relevant factorial, denominator and binomial argument is
$An+B\lambda+Cr$, with $B\in\{-1,0,1\}$. For $B\ne0$ an event occurs at
\[
 \lambda_0=(kp-An-Cr)/B.
\]
It is an integer. To find the largest cost over integer $\lambda$, it
suffices to test the range endpoints and each event together with its
integer neighbours. Between these events every floor and denominator
predicate is constant.

At a tested neighbour, another argument satisfies
\[
 Dn+E(\lambda_0+\delta)+Fr
 =\left(D-\frac{EA}{B}\right)n+\frac EBkp
   +\left(F-\frac{EC}{B}\right)r+E\delta,
 \qquad\delta\in\{-1,0,1\}.
\]
The coefficient of $kp$ is integral. Exact determinants of the affine
coefficients give comparisons of $dn+e$ with multiples of $p$, where
$|d|\le14$ and $|e|\le6$, including the neighbouring integer and all six
shifts. The checkers reconstruct and verify these bounds.

If $d\ne0$, $p\nmid E_n$ implies
$\operatorname{dist}(dn,p\Z)\ge7$, so the bounded perturbation cannot cross
an unshifted event. If $d=0$, the sign of $e$ decides the comparison,
including coincident events that separate under a shift. This also controls
whether a tested neighbour belongs to the allowed $\lambda$-range.

The remaining boundaries in $x=n/p\in(0,4)$ lie among $x=i/d$, with
\[
 d\in\{1,2,3,4,5,6,7,8,9,10,12,13,14\}.
\]
There are 216 open cells. Boundary cases imply $p\mid E_n$ and are excluded.
On each cell the integer-event patterns are constant after retaining the
signed bounded perturbations. The prescribed witnesses are
\begin{center}
\begin{tabular}{cc@{\qquad}cc}
\toprule
$\{x\}$ & expansion & $\{x\}$ & expansion\\
\midrule
$(1/6,1/5)$&A&$(1/2,4/7)$&A\\
$(1/5,1/4)$&B&$(2/3,5/7)$&A\\
$(1/4,2/7)$&A&$(3/4,4/5)$&A\\
$(1/3,2/5)$&A&$(4/5,5/6)$&B\\
$(2/5,3/7)$&B&$(5/6,6/7)$&B\\
\bottomrule
\end{tabular}
\end{center}
The exact program \texttt{verify\_arithmetic.py} checks these witnesses on
all four periods and all six shifts. It uses scaled rational midpoints
$(n_*,p_*)=(1000N,1000D)$ of each cell, where $N/D$ is reduced. No primality
of $p_*$ is asserted or needed: this is an integer floor calculation, and
nonzero distances to unshifted events are at least 1000. There are 600
nontrivial cell/shift witnesses and 93,432 integer-event tests. Every cost
is at most $2-\mu(\{x\})$. A separate ordered-infinitesimal implementation
reproduces the same result without using the scaled representatives.
\end{proof}

Since $\vp(D_n)=2$ at every prime in $\Delta_n$, Lemma~\ref{t:lem:stability}
and the baseline integrality prove Lemma~\ref{t:lem:arithmetic}.
The exclusion is substantive: at $n=1000$, $p=4001$, $r=2$, both expansions
have worst cost two while the requested cost is one. That prime divides
$4n+1$, so it is correctly excluded by $E_n$.

\subsubsection{The exponential arithmetic cost}
The prime number theorem, applied to each interval in each of the four
periods, gives
\begin{equation}\label{t:eq:saving}
 \lim_{n\to\infty}\frac{\log\Delta_n}{n}
 =S_4:=\sum_{(a,b),w}w\sum_{q=0}^{3}
        \left(\frac1{q+a}-\frac1{q+b}\right)
 =\frac{324954395039}{80313433200}.
\end{equation}
Indeed $q+a<n/p<q+b$ corresponds to a prime interval with endpoints
$n/(q+b)$ and $n/(q+a)$. Endpoint conventions do not change the limit.
Removing prime divisors of $E_n$ loses at most $2\log E_n=O(\log n)$,
since $E_n$ is a polynomial of fixed degree 182. Thus
\begin{equation}\label{t:eq:C}
 C:=\lim_{n\to\infty}\frac{\log K_n}{n}
 =15+3\log3+2\log4-S_4
 =17.022347838961690145\ldots.
\end{equation}
The decimal is descriptive; the final sign is certified rationally below.

\subsection{A compact contour and a real tail}\label{t:sec:contours}
Set
\[
 a=\frac45,\quad b=\frac74,\quad R=\frac1{26000000},\qquad
 X(t)=t+ia\,t(1-t),\quad Y(u)=u-ib\,u(1-u).
\]
Define
\begin{align}
 f(x,y)&=\frac{x^6(1-x)^6y^4(1-y)^4}{(x-y)^5(y+3)^7},
       \label{t:eq:phase-complex}\\
 f_T(t,u)&=\frac{[t(1-t)]^6[u(1-u)]^4}
 {(1-tu)^5(1+3u)^7}.
       \label{t:eq:phase-tail}
\end{align}
Continuous corner limits are set to zero.

\begin{lemma}[Exact phase bounds]\label{t:lem:phases}
For $0\le t,u\le1$,
\[
 |f(X(t),Y(u))|<R,\qquad 0\le f_T(t,u)<R.
\]
\end{lemma}
\begin{proof}[Rational certificate and justification]
Put $T=t(1-t)$, $U=u(1-u)$ and
$S_c(v)=v^2(1+c^2-2c^2v+c^4v^2)$. Direct algebra gives
\begin{equation}\label{t:eq:modulus}
 |f(X(t),Y(u))|^2
 =\frac{S_a(T)^6S_b(U)^4}
 {[(t-u)^2+(aT+bU)^2]^5[(3+u)^2+b^2U^2]^7}.
\end{equation}
The functions $S_c$ are increasing on $[0,1/4]$, since
\[
 S'_c(v)=2v(1+c^2-3c^2v+2c^4v^2)
 \ge2v(1+c^2/4)\ge0.
\]
Also $(3+u)^2+b^2u^2(1-u)^2$ is increasing: its derivative is at least
$6-b^2/2>0$.

For a rational rectangle $[t_-,t_+]\times[u_-,u_+]$, let
$T_-,T_+,U_-,U_+$ be the exact extrema of $t(1-t),u(1-u)$ there, and put
$d=\max(0,t_--u_+,u_--t_+)$. The following are rational upper bounds:
\begin{align}
 |f|^2&\le
 \frac{S_a(T_+)^6S_b(U_+)^4}
 {[d^2+(aT_-+bU_-)^2]^5
  [(3+u_-)^2+b^2u_-^2(1-u_-)^2]^7},\label{t:eq:box-complex}\\
 f_T&\le
 \frac{T_+^6U_+^4}{(1-t_+u_+)^5(1+3u_-)^7}.
 \label{t:eq:box-tail}
\end{align}
Boxes with a vanishing lower denominator are not accepted.

For $\varepsilon=1/32$, the two equal-endpoint complex corner squares
are handled analytically. On $0\le t,u\le\varepsilon$,
$t^6u^4/(t+u)^5\le\varepsilon^5/32$. The expression increases in $t$,
and after setting $t=\varepsilon$ it increases in $u\le\varepsilon$.
Using $b\ge a$, $|Y+3|\ge3$ and $S_c(v)\le v^2(1+c^2)$ gives, at
both corners,
\begin{equation}\label{t:eq:corner-C}
 |f|^2\le
 \frac{\varepsilon^{10}(1+a^2)^6(1+b^2)^4}
 {1024a^{10}(1-\varepsilon)^{10}3^{14}}<R^2.
\end{equation}
For the real-tail corner $(1,1)$, reflect the two coordinates. Since
$1-tu\ge(1-\varepsilon)((1-t)+(1-u))$ there,
\begin{equation}\label{t:eq:corner-T}
 f_T\le
 \frac{\varepsilon^5}{32(1-\varepsilon)^5(4-3\varepsilon)^7}<R.
\end{equation}
Both strict inequalities are integer cross-multiplications.

The remaining regions are covered by 1,380 complex rectangles and 74
real-tail rectangles. The files \texttt{complex\_contour\_certificate.json}
and \texttt{tail\_contour\_certificate.json} give their rational endpoints.
The checker verifies \eqref{t:eq:box-complex} and \eqref{t:eq:box-tail} on every
rectangle and verifies a disjoint cover of the entire square, including the
corner patches. It checks vertical slabs and interval tilings, not merely
the sum of areas. No numerical maximum or saddle estimate enters the check.
\end{proof}

\begin{lemma}\label{t:lem:bound}
For every $n\ge1$ and $0\le r\le5$, $|J_{n,r}|<14R^n$.
\end{lemma}
\begin{proof}
In \eqref{t:eq:vertical}, deform the ray $[0,-i\infty)$ into
$Y([0,1])\cup[1,\infty)$. For fixed $0<x<1$ the poles at $x$ and $-3$
are not crossed, and the large connecting arcs vanish by
$O(|\eta|^{-4n-2})$ decay. In the compact part, with $\eta$ below the real
axis, deform the $x$-path upward to $X([0,1])$.

For completeness, during the straight homotopy
$X_s(t)=t+isa\,t(1-t)$, $0\le s\le1$, the separation near $(0,0)$ satisfies
\[
 D:=|X_s(t)-Y(u)|\ge|t-u|,\qquad D\ge bu/2
 \quad(0\le t,u\le1/2).
\]
Thus $t+u\le(1+4/b)D$, or
\begin{equation}\label{t:eq:separation}
 |X_s(t)-Y(u)|\ge\frac7{23}(t+u).
\end{equation}
The reflected inequality holds near $(1,1)$. The local full weights are
$O((t+u)^{5n-1})$ and
$O(((1-t)+(1-u))^{5n+r-1})$. These bounds justify truncation, the
endpoint passage and the contour deformations. The real tail has the same
integrable power count at its finite joint endpoint.

Let $C_{n,r}$ be the compact double integral after deformation, without
$\sigma_n$, and let
\[
 T_{n,r}=\int_0^1\!\int_0^1
 \frac{(1-t)^r f_T(t,u)^n}{(1-tu)(1+3u)}\dd u\dd t\ge0.
\]
The substitution $\eta=1/u$ in the real tail gives
\begin{equation}\label{t:eq:split}
 F_n^{(0)}(6n+r)=\sigma_n\{C_{n,r}+(-1)^{5n+1}T_{n,r}\}.
\end{equation}
The positively oriented residue contour travels along the lower arc from
$0$ to $1$ and back along its conjugate. Therefore
\[
 F_n^{(1)}(6n+r)=\frac{\sigma_n}{\pi}\Im C_{n,r},
\]
and \eqref{t:eq:mixed} becomes the explicitly real formula
\begin{equation}\label{t:eq:real-form}
 J_{n,r}=\sigma_n\left\{\Re C_{n,r}+(-1)^{5n+1}T_{n,r}
               -\frac{\log3}{\pi}\Im C_{n,r}\right\}.
\end{equation}
In particular, the $i\pi$ term has not been discarded: it cancels the
imaginary part with the prescribed orientation.

We have $|1-X(t)|\le1$, $|X'(t)Y'(u)|<3$, and $|Y(u)+3|\ge3$.
Furthermore
\[
 I:=\int_0^1\!\int_0^1\frac{\dd t\dd u}{|X(t)-Y(u)|}
 \le5\log2+2<7.
\]
On each equal-endpoint half-square, use
$|X-Y|\ge a(t+u)/2$ or its reflection; together these contribute at most
$5\log2$. On either off-diagonal half-square the separation is at least
$1/4$: either $|t-u|\ge1/4$, or both coordinates are in $[1/4,3/4]$ and
the imaginary separation exceeds $3a/8=3/10$. The remaining contribution
is at most two.

Lemma~\ref{t:lem:phases} now yields $|C_{n,r}|<7R^n$. Also
\[
 T_{n,r}<R^n\int_0^1\!\int_0^1\frac{\dd t\dd u}{1-tu}
 =R^n\sum_{k\ge1}\frac1{k^2}<2R^n.
\]
Since $\log3<2$ and $\pi>3$, \eqref{t:eq:real-form} gives
$|J_{n,r}|<(\frac53\cdot7+2)R^n<14R^n$.
\end{proof}

\subsection{The contiguous recurrence}\label{t:sec:recurrence}
\begin{lemma}\label{t:lem:recurrence}
For $n\ge2$ and integer $H\ge6n$,
\begin{equation}\label{t:eq:recurrence}
 \sum_{j=0}^{5}c_j(H,n)F_n(H+j)=0,
\end{equation}
where
\begin{align*}
 c_0={}&-4(H+1)(H+2)(H-n+1),\\
 c_1={}&(H+2)(17H^2+39Hn+74H-68n^2+43n+81),\\
 c_2={}&-28H^3-145H^2n-242H^2-89Hn^2-760Hn-704H\\
       &\quad+364n^3-2n^2-976n-688,\\
 c_3={}&22H^3+161H^2n+238H^2+387Hn^2+1146Hn+866H\\
       &\quad-186n^3+1101n^2+2017n+1058,\\
 c_4={}&-(H+n+4)(8H^2+55Hn+70H+180n^2+266n+156),\\
 c_5={}&(H+n+4)(H+n+5)(H+2n+5).
\end{align*}
In particular, $c_5(H,n)>0$ in this range.
\end{lemma}
\begin{proof}[Exact identity and boundary terms]
Let $w=x^{6n}(1-x)^H\phi_n(x,y)$ and
\[
 p_0=x(1-x)(x-y),\qquad q_0=y(1-y)(y+3)(x-y).
\]
The supplied polynomials $P,Q\in\Z[x,y,H,n]$ satisfy the identity below.
They are recorded in the file
\begin{center}\small\texttt{negative\_contiguous\_certificate.json}.\end{center}
\begin{multline}\label{t:eq:telescoper}
 \partial_x(p_0Pw)+\partial_y(q_0Qw)
 =5n(H+n+3)(H+n+4)\,w
 \sum_{j=0}^{5}c_j(H+1,n)(1-x)^{j+1}.
\end{multline}
After cancelling $w$, this is a polynomial identity. The verifier
reconstructs the logarithmic derivatives from the actual weight with pole
$y=-3$ and checks the identity exactly. It also checks the six cleaned
coefficients and $\deg_yP\le4$, $\deg_yQ\le2$. No recurrence fitting is used.

For the vertical integral, the $x$ flux vanishes at $x=0,1$ for integer
$H\ge0$. The $y=0$ flux vanishes; at the joint origin both fluxes are
$O((x+|y|)^{5n+1})$. At infinity the $x$ flux is
$O(|y|^{3-4n})$ and the $y$ flux is $O(|y|^{4-4n})$, respectively
integrable and vanishing for $n\ge2$. For the residue component use a fixed
closed contour away from the endpoints. Its $y$ flux is a total derivative
of a single-valued rational function; the $x$ boundary terms again vanish.

Hence the raw identity integrates to zero for both components for $n\ge2$,
$H\ge0$. Divide by $5n(H+n+3)(H+n+4)$ and replace $H$ by $H-1$.
The resulting recurrence is \eqref{t:eq:recurrence}; no zero denominator is
introduced for the stated range.
\end{proof}

\subsection{A fixed-\texorpdfstring{$n$}{n} endpoint asymptotic and nonvanishing}
\begin{proposition}\label{t:prop:endpoint}
For each fixed integer $n\ge1$,
\begin{equation}\label{t:eq:endpoint}
 F_n(H)\sim\frac{(4n)!(n-1)!}{(-3)^{7n+1}}\,H^{-5n-1}
 \qquad(H\to\infty).
\end{equation}
In particular, its tail is nonzero. No uniformity in $n$ is asserted or
needed.
\end{proposition}
\begin{proof}
Write
\[
 g_n(x)=\sigma_n\int_0^{-i\infty}\phi_n(x,\eta)\dd\eta,
 \qquad F_n^{(0)}(H)=\int_0^1x^{6n}(1-x)^H g_n(x)\dd x.
\]
The substitution $\eta=-ixt$ gives
\begin{equation}\label{t:eq:scaled-inner}
 x^ng_n(x)=\sigma_n(-i)\int_0^\infty
 \frac{t^{4n}(1+ixt)^{4n}}
 {(1+it)^{5n+1}(3-ixt)^{7n+1}}\dd t.
\end{equation}
For $0<x\le1$, the modulus of the $x$-dependent ratio is bounded by
\[
 \frac{(1+x^2t^2)^{2n}}{(9+x^2t^2)^{(7n+1)/2}}
 \le3^{-3n-1}.
\]
The remaining dominating function has tail $t^{-n-1}$, integrable for
$n\ge1$. Dominated convergence applies. Rotation in the beta integral gives
\[
 (-i)\int_0^\infty\frac{t^{4n}}{(1+it)^{5n+1}}\dd t
 =-\Bfun(4n+1,n).
\]
Consequently
\begin{equation}\label{t:eq:inner-limit}
 \lim_{x\downarrow0}x^ng_n(x)
 =-\sigma_n3^{-7n-1}\Bfun(4n+1,n)
 =\frac{(4n)!(n-1)!}{(-3)^{7n+1}(5n)!}.
\end{equation}
The function $b_n(x)=x^ng_n(x)$ is bounded on $(0,1]$ and has this limit.
Substitute $x=t/H$ in
$F_n^{(0)}(H)=\int_0^1x^{5n}(1-x)^Hb_n(x)\dd x$, and use
$(1-t/H)^H\le e^{-t}$. A second dominated-convergence argument shows
\[
 H^{5n+1}F_n^{(0)}(H)\longrightarrow (5n)!b_n(0).
\]

The residue in \eqref{t:eq:residue} is
\[
 A_n(x)=\Res_{y=x}\phi_n(x,y)
 =\frac{(-1)^{5n+1}}{(5n)!}\frac{d^{5n}}{dx^{5n}}
 \left(\frac{x^{4n}(1-x)^{4n}}{(x+3)^{7n+1}}\right).
\]
It is a bounded rational function on $[0,1]$. Therefore
\[
 F_n^{(1)}(H)=O\left(\int_0^1x^{6n}(1-x)^H\dd x\right)
 =O(H^{-6n-1}),
\]
which is smaller than the leading term by a factor $O(H^{-n})$.
Equations \eqref{t:eq:mixed} and \eqref{t:eq:inner-limit} prove
\eqref{t:eq:endpoint}. The sign is $(-1)^{n+1}$, not always positive.
\end{proof}

\begin{corollary*}
For $n\ge2$, five consecutive values $F_n(H),\ldots,F_n(H+4)$ with
$H\ge6n$ cannot all vanish. In particular, at least one of
$J_{n,0},\ldots,J_{n,5}$ is nonzero.
\end{corollary*}
\begin{proof}
A block of five zeros and the nonzero forward coefficient in
Lemma~\ref{t:lem:recurrence} would force every subsequent value to vanish.
This contradicts Proposition~\ref{t:prop:endpoint} for that same fixed $n$.
\end{proof}

\subsection{The exact positive margin and final contradiction}
The constants in \eqref{t:eq:C} and Lemma~\ref{t:lem:bound} give
\begin{align}
 \gamma&=\log26000000-C\notag\\
 &=3\log2+6\log5+\log13-3\log3-15+S_4\label{t:eq:gamma}\\
 &=0.0512592570240660045\ldots.\notag
\end{align}
There is a short rational sign certificate. The bounds
\[
 \log2>0.6931,\quad\log5>1.6094,\quad\log13>2.5649,
 \quad\log3<1.0987,\quad S_4>4.046
\]
are inequalities between exact rational numbers and logarithms, not rounded
sign decisions. They follow from the positive atanh series: for
$t=(z-1)/(z+1)$,
\[
 L_N(z)=2\sum_{k=0}^{N-1}\frac{t^{2k+1}}{2k+1}
 <\log z< L_N(z)+\frac{2t^{2N+1}}{(2N+1)(1-t^2)}.
\]
The checker uses $N=80$. Substitution in \eqref{t:eq:gamma} yields
\begin{equation}\label{t:eq:gamma-rational}
 \gamma>
 3\frac{6931}{10000}+6\frac{16094}{10000}
 +\frac{25649}{10000}-3\frac{10987}{10000}-15+
 \frac{4046}{1000}
 =\frac{101}{2000}>\frac1{20}.
\end{equation}

\begin{proof}[Proof of the $-1/3$ assertion in Theorem~\ref{thm:main}]
For all sufficiently large $n$, the six forms $K_nJ_{n,r}$ have integer
coefficients in $1,\theta$, at least one is nonzero, and
\[
 \max_{0\le r\le5}|K_nJ_{n,r}|
 <14\exp\{-(\gamma+o(1))n\}\longrightarrow0.
\]
If $\theta=a/b\in\Q$ with $b\ge1$, every nonzero integer linear form in
$1,\theta$ has modulus at least $1/b$. This contradicts the displayed
limit and proves the $-1/3$ assertion in Theorem~\ref{thm:main}.
\end{proof}

\section{Verification, stress tests, and scope}\label{sec:verification-combined}
The theorem-level arguments above separate exact finite certificates from the deductions that extend them to all sufficiently large $n$.  The accompanying archive contains the original proof verifiers for the three constructions, together with independently implemented audit routines and the supplementary $10{,}000$-term calculations described below.

\subsection{Proof-critical finite certificates}
For $-1/4$, the finite certificate checks the rational interval partition for the retained Rhin--Viola prime divisor, the exact constant in its three-period truncation, the fixed-contour inequality and the displayed recurrence identities.  For $1/5$, the core checks cover 108 floor cells for all six shifts, 458 complex and 205 real contour rectangles, and the polynomial telescoping identity that gives the six-term recurrence.  For $-1/3$, the corresponding arithmetic certificate has 216 rational cells and 93,432 event tests, while the contour certificate has 1,380 compact rectangles and 74 tail rectangles.  In all cases the final sign of the exponential margin is decided by rational inequalities rather than floating-point comparison.

A separately implemented audit reproduced the canonical linear forms for all shifts through $n=20$ for the three values, including the negative branch by algebraic continuation of the coefficient polynomials.  It also rechecked the two source normalisations that are easiest to misread in the six-parameter formulas, the cancellation of the unwanted logarithmic term, the summand-level factorial divisors at several values $n\ge1000$, the recurrences on true coefficient values, the endpoint signs and the contour box inequalities.  These checks are supplementary: in particular, the all-$n$ floor-stability arguments in the $1/5$ and $-1/3$ proofs remain mathematical deductions to be reviewed as such.

For orientation, an independent saddle calculation gives exponential decay rates approximately
\[
 -23.499,\qquad -18.906,\qquad -17.101
\]
for the three constructions in the order $-1/4,1/5,-1/3$, compared with the deliberately weaker certified bounds approximately $-23.383,-18.826,-17.074$.  These numerical values are not used in the proofs.

\subsection{Ten thousand exact primitive forms for the two crossbreed constructions}
As a supplementary stress test, we propagated the first $10{,}000$ canonical coefficient pairs for each of $\Li_2(1/5)$ and $\Li_2(-1/3)$.  Separate diagonal recurrences in $n$, each of order five with polynomial coefficients of degree 46, were reconstructed exactly over $\Q$ and checked against direct GMP coefficient extraction.  High-index hold-out checks at
\[
 97,\ 499,\ 997,\ 1999,\ 4999,\ 9995
\]
were also performed modulo three unrelated large primes before the long propagation was used.

At every $1\le n\le10000$, the canonical rational pair was cleared to an integer pair and its complete coefficient gcd was removed.  The resulting primitive form was then bounded by the rigorous contour estimate from the corresponding proof.  All $20{,}000$ primitive forms were certified to have modulus less than one.  Moreover, in each sequence all $9{,}999$ determinants of adjacent primitive coefficient pairs are nonzero, so two consecutive forms cannot vanish simultaneously.  At $n=10000$ the conservative bounds are already
\[
 |L_n(1/5)|<2^{-5837.28\ldots},\qquad
 |L_n(-1/3)|<2^{-2668.27\ldots}.
\]
This is strong finite evidence, not a replacement for the all-$n$ divisor, contour or nonvanishing proofs.  The full $20{,}000$-row data, exact recurrences, seeds and regeneration code are included in the supplementary archive.  The $-1/4$ verification archive also contains an earlier independent $10{,}000$-term exact stress test for its five-parameter sequence.

\subsection{Review boundary}
The most important remaining objects for independent mathematical vetting are the coefficient-level factorial transfer inside the two Viola--Zudilin expansions, the floor-stability passage from exact cells to all sufficiently large $n$, and the contour homotopies and boundary terms in the recurrence arguments.  The $-1/3$ proof has the smallest certified margin and should consequently be regarded as the most sensitive to a hidden normalisation or sign error.  No such error was found in the source-level or computational audits described above.

\section*{Statement on computational and AI assistance}
OpenAI's ChatGPT was used extensively in exploratory parameter searches, symbolic and numerical calculations, literature search, proof auditing, software preparation and drafting.  Anthropic's Claude was used separately for adversarial source-checking and independent computational vetting of the three proof drafts.  The mathematical arguments are intended to stand independently of that assistance; the author takes responsibility for the contents.  Exact executable certificates are supplied to make finite computations reproducible, but AI-assisted checking is not a substitute for independent specialist review or proof-assistant formalisation.

\end{document}